\documentclass[oneside, english]{amsart}
\usepackage[T1]{fontenc}
\usepackage[latin9]{inputenc}
\usepackage{geometry}
\usepackage{verbatim}
\usepackage{amsthm}
\usepackage{amstext}
\usepackage{amssymb}
\usepackage{graphicx}
\usepackage{esint}
\usepackage{url}
\usepackage{yhmath}
\usepackage[all]{xy}
\usepackage{color}
\usepackage{bbm}

\usepackage[percent]{overpic}

\usepackage[space]{grffile} 
\makeatletter
\numberwithin{equation}{section}
\numberwithin{figure}{section}
\theoremstyle{plain}
\newtheorem*{thm*}{\protect\theoremname}
\theoremstyle{plain}
\newtheorem{thm}{\protect\theoremname}[section]
\theoremstyle{plain}
\newtheorem{lem}[thm]{\protect\lemmaname}
\theoremstyle{remark}
\newtheorem{remark}[thm]{\protect\remarkname}
\theoremstyle{plain}
\newtheorem{proposition}[thm]{\protect\propositionname}
\theoremstyle{plain}

\theoremstyle{plain}

\theoremstyle{plain}

\theoremstyle{plain}

\theoremstyle{plain}

\theoremstyle{plain}

\theoremstyle{plain}

\usepackage{mathrsfs}
\usepackage{subfigure}

\newcommand{\SLE}{\mathrm{SLE}}

\newcommand{\PR}{\mathbb{P}}
\newcommand{\EX}{\mathbb{E}}
\newcommand{\limPR}{\mathbf{P}}
\newcommand{\limEX}{\mathbf{E}}

\newcommand{\sZ}{\mathcal{Z}}

\newcommand{\bR}{\mathbb{R}}
\newcommand{\R}{\bR}

\newcommand{\bZ}{\mathbb{Z}}
\newcommand{\Z}{\bZ}
\newcommand{\bN}{\mathbb{N}}
\newcommand{\N}{\bN}

\newcommand{\bC}{\mathbb{C}}
\newcommand{\C}{\bC}

\newcommand{\bH}{\mathbb{H}}

\newcommand{\domain}{\Lambda}
\newcommand{\bdry}{\partial}

\newcommand{\confmap}{\phi} 
\newcommand{\confmapH}{\confmap} 

\newcommand{\ud}{\mathrm{d}}

\newcommand{\Arch}{\mathrm{LP}}
\newcommand{\LP}{\Arch}

\newcommand{\Mmat}{\mathscr{M}}

\newcommand{\PartF}{\sZ}

\newcommand{\HarmMeas}{\mathsf{H}}

\newcommand{\PoissonK}{\mathsf{P}}
\newcommand{\PoissonKH}{\mathcal{P}}

\newcommand{\ExcK}{\mathsf{K}}
\newcommand{\ExcKH}{\mathcal{K}}

\newcommand{\FominDet}{\mathbf{\Delta}}
\newcommand{\LPdet}[2]{\FominDet_{#1}^{#2}}

\newcommand{\eps}{\varepsilon}

\newcommand{\Gr}{\mathcal{G}}
\renewcommand{\Vert}{\mathcal{V}}
\newcommand{\Edg}{\mathcal{E}}
\newcommand{\GrH}{\mathcal{H}}
\newcommand{\EdgH}{\Edg (\GrH)}
\newcommand{\VertH}{\Vert (\GrH)}

\newcommand{\ein}{e_\mathrm{in}}
\newcommand{\eout}{e_\mathrm{out}}

\newcommand{\pin}{p_\mathrm{in}}
\newcommand{\pout}{p_\mathrm{out}}

\newcommand{\xin}{x_\mathrm{in}}
\newcommand{\xout}{x_\mathrm{out}}

\newcommand{\tree}{\mathcal{T}}

\newcommand{\pathfromto}[2]{#1 \rightsquigarrow #2}

\newcommand{\edgeof}[2]{{\langle #1 , #2 \rangle}}

\newcommand{\EdgeWeight}{\mathsf{w}}
\newcommand{\mesh}{\delta}
\newcommand{\shrinkto}{\to}

\newcommand{\InfiniteGr}{\Gamma}
\newcommand{\GreenK}{\mathsf{G}}
\newcommand{\GreenKH}{\mathcal{G}}

\newcommand{\HalfAngle}{\theta}

\newcommand{\Walks}{\mathscr{W}}
\newcommand{\Mart}{\mathscr{M}}

\makeatletter
\newcommand*{\centerfloat}{%
  \parindent \z@
  \leftskip \z@ \@plus 1fil \@minus \textwidth
  \rightskip\leftskip
  \parfillskip \z@skip}
\makeatother

\AtBeginDocument{

}

\makeatother

\usepackage{babel}
\providecommand{\corollaryname}{Corollary}
\providecommand{\lemmaname}{Lemma}
\providecommand{\propositionname}{Proposition}
\providecommand{\remarkname}{Remark}
\providecommand{\theoremname}{Theorem}
\providecommand{\conjecturename}{Conjecture}

\definecolor{kallecol}{rgb}{.75,.0,.55}

\begin{document}



\

\vspace{2.5cm}

\begin{center}
\LARGE \bf \scshape {
UST branches, martingales, and multiple SLE(2)
}
\end{center}

\vspace{0.75cm}

\begin{center}
{\large \scshape Alex Karrila}\\
{\footnotesize{\tt karrila@ihes.fr; alex.karrila@gmail.com}}\\
{\small{Institut des Hautes \'{E}tudes Scientifiques}}\\
{\small{35 Route de Chartres, 91440 Bures-sur-Yvette, France}}\bigskip{}
\end{center}

\vspace{0.75cm}

\begin{center}
\begin{minipage}{0.85\textwidth} \footnotesize
{\scshape Abstract.} 
We identify the local scaling limit of multiple boundary-to-boundary branches in a uniform spanning tree (UST) as a local multiple $\SLE(2)$, i.e., an $\SLE(2)$ process weighted by a suitable partition function. By recent results, this also characterizes the ``global'' scaling limit of the full collection of full curves. The identification is based on a martingale observable in the UST with $N$ branches, obtained by weighting the well-known martingale in the UST with one branch by the discrete partition functions of the models. The obtained weighting transforms of the discrete martingales and the limiting SLE processes, respectively, only rely on a discrete domain Markov property and (essentially) the convergence of partition functions. We illustrate their generalizability by sketching an analogous convergence proof for a boundary-visiting UST branch and a boundary-visiting $\SLE(2)$.
\end{minipage}
\end{center}

\vspace{0.75cm}
\tableofcontents

\bigskip{}

\addtocontents{toc}{\protect\setcounter{tocdepth}{1}}

\section{Introduction}

Schramm--Loewner evolution (SLE) type curves are conformally invariant random curves~\cite{Schramm-LERW_and_UST, RS-basic_properties_of_SLE, GRS-SLE_in_sc_domains}, known or conjectured to describe the scaling limits of random interfaces in many critical planar lattice models~\cite{Smirnov-critical_percolation, LSW-LERW_and_UST, SS05, CN07, Zhan-scaling_limits_of_planar_LERW, SS09, HK-Ising_interfaces_and_free_boundary_conditions, CDHKS-convergence_of_Ising_interfaces_to_SLE, Izyurov-Smirnovs_observable_for_free_boundary_conditions, GW18}. A particularly interesting variant is the local multiple SLE~\cite{Dubedat-commutation, KP-pure_partition_functions_of_multiple_SLEs}, which explicitly connects SLEs to Conformal field theory, the physics description of scaling limits of critical models~\cite{BBK-multiple_SLEs, Graham-multiple_SLEs, Dubedat-SLE_and_Virasoro_representations_fusion, KKP2, Peltola-towards_a_CFT_for_SLEs}. The main result of this article, Theorem~\ref{thm: main result}, proves local multiple SLE convergence for multiple boundary-to-boundary branches in a uniform spanning tree (UST) model on $\Z^2$ (see Figure~\ref{fig: UST}), as well as its natural generalization to other isoradial lattices.

The local multiple SLE convergence of multiple UST branches was predicted in~\cite[Conjecture~4.3]{KKP} (see also~\cite[Section~2]{Dubedat-commutation} and~\cite[Section~5.1]{KW-boundary_partitions_in_trees_and_dimers}). A proof outline for that conjecture was given by the author in~\cite[Theorem~6.8]{mie2} together with several consequences, most importantly determining the scaling limit of \emph{full} curves based on the conjectured \emph{local} limit. This paper fills the omitted part~\cite[Assumption~5.1]{mie2} of that proof outline, thus completing the proof.

Scaling limits of multiple chordal interfaces have been recently also studied in terms of the global multiple SLEs~\cite{LK-configurational_measure, Lawler-glob_NSLE, Wu17, PW, BPW}. 
The proof of our main theorem~\ref{thm: main result} provides an important example of the relation between convergence proofs based on local and global multiple SLEs, as discussed in~\cite[Section~1]{mie2}: On the one hand, due to recent characterization results for global multiple SLEs, rather short convergence proofs can nowadays be given for various discrete chordal curve models, when conditioned on the pairing of the boundary points by the curves~\cite{Wu17, BPW}. Such proofs require as an input the convergence of the corresponding one-curve model to chordal SLE (see~\cite{Zhan-scaling_limits_of_planar_LERW} on the UST). To extend such proofs to unconditional models, one in addition needs to solve the scaling limit probabilities of the different pairings of boundary points. This is done for some lattice models in~\cite{Smirnov-critical_percolation, PW18} and Theorem~\ref{thm: scaling limit WST connectivity probas} of this paper.
In conclusion, using~\cite{Zhan-scaling_limits_of_planar_LERW}, global multiple SLE theory, and Theorem~\ref{thm: scaling limit WST connectivity probas}, one could thus characterize the scaling limit of (unconditional) multiple UST branches in terms of global multiple SLEs.

In this paper, we instead convert the convergence proof of~\cite{Zhan-scaling_limits_of_planar_LERW}, based on martingale observables, from one to multiple UST branches. Key tools are a discrete Girsanov transform, converting discrete martingales from one to multiple UST branches,
and (essentially) Theorem~\ref{thm: scaling limit WST connectivity probas}, establishing the convergence of the conversion factors in such martingale transforms. Compared to using global multiple SLEs, this approach roughly speaking takes fewer inputs, but with the price of re-doing the input from~\cite{Zhan-scaling_limits_of_planar_LERW}.
The discrete Girsanov transform is not specific to the UST model and, if an analogue of Theorem~\ref{thm: scaling limit WST connectivity probas} were at hand, SLE convergence proofs could be promoted from one to multiple curves similarly in other lattice models; see the use of~\cite{Smirnov-critical_percolation} in~\cite{KS18} for comparison. 
We illustrate this generalizability by showing how to extend our proof to a boundary-visiting UST branch and boundary-visiting $\SLE(2)$.

Martingale arguments proving the convergence lattice interfaces to different multiple SLE type curves are given in~\cite{Izyurov-critical_Ising_interfaces_in_multiply_connected_domains, KS-bdary_loops_FK, KS18, mie2, Izyurov-FK_multiple_SLE}. The convergence of different variants of a single UST branch, or the closely related loop-erased random walk, to different SLE variants has been proven in~\cite{LSW-LERW_and_UST, Zhan-scaling_limits_of_planar_LERW, LV-natural_parametrization_for_SLE, CW-mLERW}, and for isoradial and even more general graphs in~\cite{CS-discrete_complex_analysis_on_isoradial, YY-Loop-erased_random_walk_and_Poisson_kernel_on_planar_graphs, Suzuki-Convergence_of_LERW_on_a_planar_graph_to_a_chordal_SLE(2), Uchi}.

\subsection*{Organization} Section~\ref{sec: setup and statement} gives the precise statement of the main result and a brief discussion of its consequences. The following three sections constitute the proof: Section~\ref{sec: connectivity probas} solves the discrete partition functions and martingales in a purely combinatorial setup, Section~\ref{sec: obs conv} establishes the convergence of these observables, and Section~\ref{sec: proof of main thm} identifies the scaling limit process via the limiting martingale observable. Some technical details are postponed to Appendices~\ref{app: disc harm} and~\ref{app: boundedness pf}. The analogue of the main result for a boundary-visiting UST branch is discussed in Section~\ref{se: bdry visiting branch}, and its (non-rigorous) interpretation in terms of boundary-visiting SLEs in Appendix~\ref{app: bdry vis SLE}.

\subsection*{Acknowledgements}
The author has benefited from useful discussions and correspondence with Dmitry Chelkak, Hugo Duminil-Copin, Konstantin Izyurov, Antti Kemppainen, Eveliina Peltola, Lauri Viitasaari, Hao Wu, and especially Kalle Kyt\"{o}l\"{a}, who also provided a simulation code for UST figures. The author also wishes to thank the anonymous referee of this paper for comments and observations that helped to improve it.
Finally, financial support from the Vilho, Yrj\"{o} and Kalle V\"{a}is\"{a}l\"{a} foundation and the ERC grant 757296 ``Critical Behavior of Lattice Models (CriBLaM)'' is gratefully acknowledged.

\addtocontents{toc}{\protect\setcounter{tocdepth}{2}}


\bigskip

\section{Setup and statement}
\label{sec: setup and statement}

This section introduces the precise setup and statement of the main result. The combinatorial model is defined Subsection~\ref{subsec: combinatorial models}. Section~\ref{subsec: isoradial graphs} introduces isoradial graphs on which the scaling limit results are obtained. The necessary background on Loewner evolutions and (multiple) SLEs are reviewed in Section~\ref{subsec: LE and SLE}, and in Section~\ref{subsec: statement} we are ready to state and discuss the theorem.

\subsection{The weighted spanning tree and its boundary-to-boundary branches}
\label{subsec: combinatorial models}

\subsubsection{The random spanning tree model}

Let $\GrH = (\VertH, \EdgH)$ be a connected finite graph. A \emph{spanning tree} of $\GrH$
is a subgraph $\tree$ that is connected and acyclic (is a tree) and contains all the vertices of $\GrH$ (is spanning). Endow the edges $\EdgH$ with positive weights $\EdgeWeight: \EdgH \to \R_{>0}$.
The \emph{weighted (random) spanning tree} on $\GrH$ is a random spanning tree with probabilities
\begin{equation*}
\PR [ \tree  ] \propto \prod_{e \in \tree} \EdgeWeight (e)
\end{equation*}
Note that if all edges $e \in \EdgH$ carry equal weight, this becomes a uniform random spanning tree.

\begin{figure}
\centering
\begin{overpic}[width=0.35\textwidth]{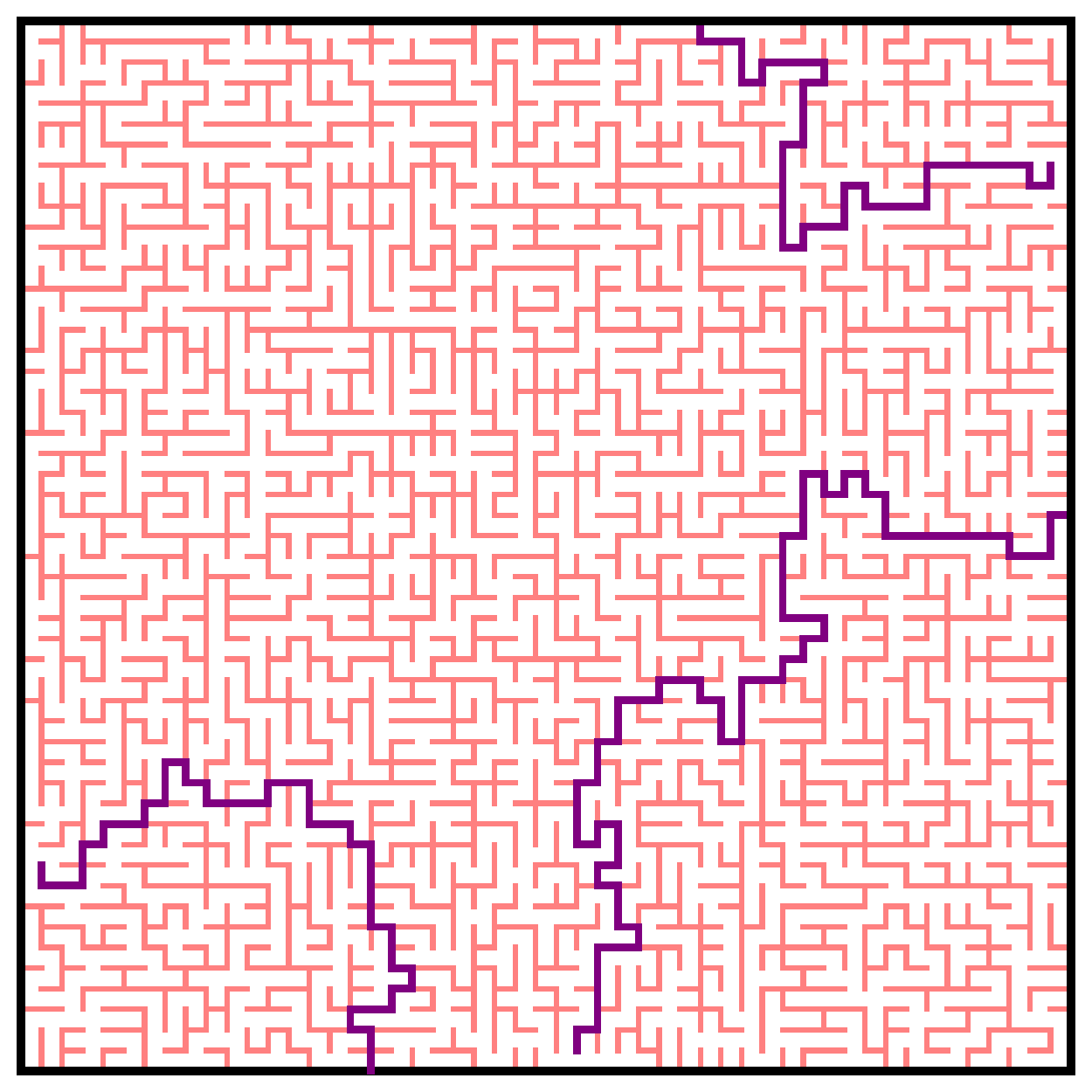}
 \put (-7,18) {\large $e_1$}
 \put (32,-5) {\large $e_2$}
  \put (52,-5) {\large $e_3$}
 \put (100,51) {\large $e_4$}
  \put (100,84) {\large $e_5$}
  \put (63,102) {\large $e_6$}
\end{overpic}
\caption{A WST sample on a $50 \times 50$ square grid graph, with the boundary branches from the interior vertices of the odd edges $e_1$, $e_3$, and $e_5$ reaching the boundary each via a different even edge $e_2$, $e_4$, or $e_6$.}
\label{fig: UST}
\end{figure}
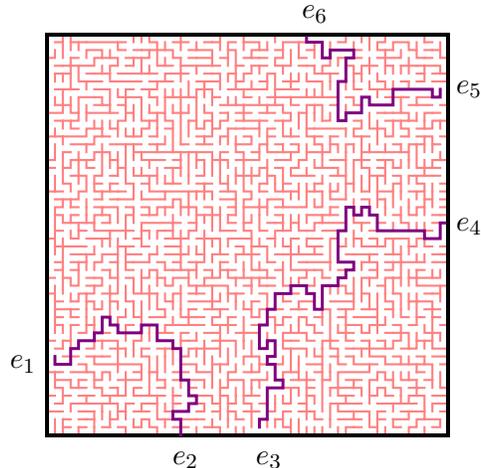

We will in this paper always study the planar weighted spanning tree with wired boundary conditions, meaning the following. Let $\Gr=(\Vert, \Edg)$ be a finite connected planar graph with a fixed planar embedding. Declare some vertices adjacent to the infinite face of $\Gr$ 
as boundary vertices $\Vert^\bdry$, and the remaining vertices of $\Vert$ as interior vertices $\Vert^\circ$. The \emph{weighted spanning tree with wired boundary conditions (WST) on $\Gr$} is then a weighted spanning tree on the graph $\GrH = \Gr/\bdry$ obtained by identifying all the boundary vertices $\Vert^\bdry$ to a single vertex $v^\bdry$. For notational simplicity, we will identify the edges of $\Gr$ and $\Gr/\bdry$, hence both graphs endowed with the edge weights $\EdgeWeight$, and we regard the WST as a subgraph of both $\Gr/\bdry$ and $\Gr$ via this identification. Edges between the interior and boundary vertices $\Vert^\circ$ and $\Vert^\bdry$ are called \emph{boundary edges} $\bdry \Edg$.

\subsubsection{Boundary-to-boundary branches}

Note that in a WST tree, each interior vertex $v \in \Vert^\circ$ connects to the boundary vertices $\bdry \Vert$ by a unique simple path, called the \emph{boundary branch from $v$}. Let $e_1, \ldots, e_{2N} \in \bdry \Edg$ be distinct boundary edges, indexed counterclockwise along the boundary. Condition the WST on $\Gr$ on the event that the boundary branches from the interior vertices of the odd edges $e_1, e_3, \ldots, e_{2N - 1}$ reach the boundary $\bdry \Vert$ via the even edges $e_2, e_4, \ldots, e_{2N}$, each using a different even edge; see Figure~\ref{fig: UST} for illustration. (We assume that $\Gr$ and $e_1, \ldots, e_{2N}$ are such that this conditioning is possible.) This produces $N$ boundary branches in the WST, and adding the odd edges $e_1, e_3, \ldots, e_{2N - 1}$, we obtain $N$ chordal, vertex-disjoint simple paths on $\Gr$. These chordal vertex-disjoint paths are called \emph{WST boundary-to-boundary branches pairing $e_1, \ldots, e_{2N}$}.

The way the boundary-to-boundary branches pair the edges $e_1, \ldots, e_{2N}$ is encoded into a partition $\alpha$ of the set $\{ 1, 2, \ldots, 2N \}$ into disjoint pairs. Due to planarity and the disjointness of the boundary-to-boundary branches, $\alpha$ is a \emph{planar pair partition} a.k.a. a \emph{link pattern}, i.e., the pairs of $\alpha$ among the real-line points $\{ 1, 2, \ldots, 2N \}$ can be connected by $N$ disjoint curves in the upper half-plane. The set of link patterns on $\{ 1, 2, \ldots, 2N \}$ is denoted by $\LP_N$.

\subsection{Isoradial graphs}
\label{subsec: isoradial graphs}

\subsubsection{Isoradial lattices}

Let $\InfiniteGr$ be an infinite, locally finite planar graph embedded in the plane.
We say that $\InfiniteGr$ an \emph{isoradial lattice} with \emph{mesh size} $\delta$ if the following holds: the vertices adjacent to each face of $\InfiniteGr$ lie on the arc of a circle with radius $\mesh$, centered inside that face. We draw the dual graph $\InfiniteGr^*$ of $\InfiniteGr$ so that the dual vertices lie at these circle center points. The four endpoints of an edge $e$ and its dual $e^*$ then determine a rhombus of side length $\mesh$. We endow isoradial lattices with edge weights 
\begin{align*}
\EdgeWeight(e) = \tan \HalfAngle_e,
\end{align*}
where $\HalfAngle_e$ is half of the opening angle of the rhombus, as divided by $e$. As in \cite{CS-discrete_complex_analysis_on_isoradial}, we assume that the half angles $\HalfAngle_e$ are bounded uniformly away from $0$ and $\pi/2$: there exists $\eta > 0$ such that
\begin{align}
\label{eq: unif bound on half-angles}
\eta \le \HalfAngle_e \le \pi / 2 - \eta
\end{align}
for all edges $e$ of $\InfiniteGr$. When studying the scaling limit $\delta \downarrow 0$, we will always assume that same $\eta$ applies for all $\delta$.

\subsubsection{Simply-connected subgraphs}

Let $\Lambda_\Gr \subset \C$ be a bounded simply-connected domain, whose boundary consists of edges and vertices in $\InfiniteGr$. Let $\Gr=(\Vert, \Edg)$ be the planar graph with $\Vert = \Vert(\InfiniteGr) \cap \overline{\domain_\Gr}$ and $\Edg = \Edg(\InfiniteGr) \cap \overline{\domain_\Gr}$. We call $\Gr$ the \emph{simply-connected subgraph} of $\InfiniteGr$. We will always run the WST model on $\Gr$ with the isoradial edge weights and boundary vertices $\Vert^\bdry = \Vert \cap \bdry \domain_\Gr $.

\subsection{Loewner evolutions and SLE}
\label{subsec: LE and SLE}

We now briefly review Loewner evolutions and SLE in the upper half-plane $\bH$. We refer the reader to the textbooks~\cite{Lawler-SLE_book, Berestycki-SLE_book, Kemppainen-SLE_book} for more details.

\subsubsection{Loewner evolutions}

The \emph{Loewner (differential) equation} in $\bH$ determines a family of complex analytic mappings $g_t$, $t \ge 0$ by
\begin{align}
\nonumber
g_0 (z) &=z \qquad \text{for all } z \in \bH \\
\label{eq: Loewner ODE}
\partial_t g_t (z) &= \frac{2}{g_t(z) - W_t},
\end{align}
where $W_\cdot:\R_{ \ge 0} \to \R$ is a given continuous function, called the \emph{driving function}. For a given $z \in \bH$, the solution $g_t(z)$ of this equation is defined up to the (possibly infinite) hitting time  $\tau(z)$ of $0$ by the process $\vert g_t(z) - W_t \vert$. The set where $g_t$ is not defined is denoted by
\[
K_t =  \{ z \in \bH : \tau(z) \le t  \} .
\]
The sets $K_t$ are growing in $t$, and for all $t$ they turn out to be \emph{hulls}, i.e., $K_t$ are bounded and closed in $\bH$, and $H_t := \mathbb{H} \setminus K_t $ is simply-connected.
It also holds true that $g_t$, called the \emph{mapping-out function} of $H_t$, is a conformal map $H_t \to \bH$ such that
\begin{align*}
g_t (z) = z + \frac{2t}{z} + O(1/z^2) \qquad \text{as }z \to \infty.
\end{align*}

The Loewner differential equation thus maps a driving function $W_\cdot$ to a growing family of hulls $K_\cdot$. Conversely, a family of growing hulls $K_\cdot$ can be obtained as the hulls of some Loewner equation (after a suitable time reparametrization) if and only if the hulls satisfy the \emph{local growth property} and have a \emph{half-plane capacity} tending to infinity (see \cite{Kemppainen-SLE_book} for definitions). The families $K_\cdot$, $W_\cdot$, and $g_\cdot$ satisfying the above conditions can thus be regarded as equivalent, and we title them \emph{Loewner evolutions}. We equip the space of Loewner evolutions with the metric topology inherited from their driving functions
\begin{align*}
d(W, \tilde{W}) = \sum_{n \in \N} 2^{-n} \min \{ 1, \sup_{t \in [0,n]} \vert \tilde{W}_t - W_t \vert \},
\end{align*}
i.e., the topology of uniform convergence over compact subsets. Random Loewner evolutions will be studied in this topology. The Borel sigma  algebra $\mathscr{F}$ of this metric is, as usual, equipped with the right continuous filtration $(\mathscr{F}_t)_{t \geq 0}$ of the stopped functions $W_{\cdot \wedge t}$, i.e., $\mathscr{F}_t = \cap_{s > t} \sigma (W_{\cdot \wedge s})$.

\subsubsection{SLE type processes}

The Schramm--Loewner evolution $\SLE(\kappa)$ with parameter $\kappa > 0$ from $0$ to $\infty$ in $\bH$, for short $\SLE(\kappa)$ in $(\bH:0, \infty)$, is the random Loewner evolution driven by a scaled Brownian motion,
\begin{align*}
W_0 &= 0 \\
\ud W_t &= \sqrt{\kappa} \ud B_t.
\end{align*}
It is a chordal curve in the precise sense that, almost surely, there exists a continuous function $\gamma :\R_{ \ge 0} \to \overline{\bH}$, with $\gamma(0)=0$ and $\gamma(t) \tiny{\stackrel{t \to \infty}{\longrightarrow} } \infty$, such that $\bH \setminus K_t$ is the unbounded component of $\mathbb{H} \setminus \gamma ([0,t])$ for all $t$. 

This work concerns SLE type processes with a partition function. The starting point at time $t=0$ are $2N$ marked real points $X^{(1)}_0 < \ldots < X^{(2N)}_0 $ and a smooth, positive \emph{partition function} $\PartF: \R^{2N} \to \R_+ $. The driving function is bound to the evolution a special $j$:th marked point
\begin{align*}
W_t = X^{(j)}_t \qquad \text{for all } t,
\end{align*}
and $W_t$ and the remaining the marked real points $X^{(i)}_t = g_t (X^{(i)}_0)$, with $i \neq j$, evolve according to the SLE type stochastic differential equations
\begin{align}
\label{eq: defn of ptt fcn SLE}
\begin{cases}
\ud W_t &= \sqrt{\kappa} \ud B_t + \kappa \frac{\partial_j \PartF (X^{(1)}_t  \ldots  X^{(2N)}_t )}{\PartF (X^{(1)}_t  \ldots  X^{(2N)}_t )} \ud t \\
\ud X^{(i)}_t &= \frac{2 }{X^{(i)}_t - W_t} \ud t \qquad \text{for all } i \neq j.
\end{cases}
\end{align}

\subsubsection{Localizations}

The partition function SLEs obtained as scaling limits in this paper will be so-called local multiple SLEs. We will not need any inputs from the theory of local multiple SLEs, but it is necessary to comply with their inherently local nature.

A \emph{localization neighbourhood} is a bounded open neighbourhood $U$ of $ X^{(j)}_0 = W_0$ in $\bH$, whose closure is a hull bounded away from all the other marked (starting) points $X^{(i)}_0$, $ i \neq j$. Morally, we would like to consider the partition function SLE up to the time the hulls $K_\cdot$ exit the neighbourhood $U$. However, such an exit time is not continuous in our topology of Loewner evolutions, posing problems when studying weak convergence of lattice models. Hence, we use the \emph{continuous modification $\tau$ of the exit time of $U$}, as defined in~\cite{mie2}. The stopping at $\tau$ comes later than the exit time of $U$ but before the exit time of its $\epsilon$-thickening $U_\epsilon$, where a small $\epsilon > 0$ is chosen as an input in the definition of $\tau$. The precise definition is not important in the context of this paper.

Due to working only up to the stopping time $\tau$, it suffices to define an SLE partition function \linebreak $\PartF(x_1, \ldots, x_{2N})$ for $x_1 < \ldots < x_{2N}$.

\subsection{The main result and some consequences}
\label{subsec: statement}

The main result of this paper is given the following setup and notation.

Let $(\InfiniteGr_n)_{n \in \N}$ be a sequence of of isoradial lattices with mesh sizes $\mesh_n \downarrow 0$. Let $(\Gr_n; e_1^{(n)}, \ldots, e_{2N}^{(n)})$ be simply-connected subgraphs of $\InfiniteGr_n$ with a fixed number $2N$ of marked boundary edges. Assume that, as planar domains with marked boundary points, $\Gr_n$ are uniformly bounded and converge in the Carath\'{e}odory sense~(see, e.g.,~\cite{CS-discrete_complex_analysis_on_isoradial} for the definition) to a domain $(\domain; p_1, \ldots, p_{2N})$ with $2N$ distinct marked prime ends.

Let $\confmap_n: \domain_n \to \bH$ and $\confmap: \domain \to \bH$ conformal maps such that $\confmap_n^{-1} \to \confmap^{-1}$ uniformly over compact subsets of $\bH$. 
Such maps exist by the Carath\'{e}odory convergence, and can be chosen so that, denoting $\confmap (p_1, \ldots, p_{2N}) =(X^{(1)}_0  \ldots  X^{(2N)}_0)$, we have $ - \infty < X^{(1)}_0 < \ldots < X^{(2N)}_0 < \infty$. We also fix an index $1 \leq j \leq 2N $ and a localization neighbourhood $U$ of $X^{(j)}_0$.

Consider now WST boundary-to-boundary branches on $(\Gr_n; e_1^{(n)}, \ldots, e_{2N}^{(n)})$, mapped conformally to $\bH$ by the maps above. Let $W^{(n)}_\cdot$ denote the driving functions in the Loewner evolutions describing the growth of the boundary-to-boundary branch starting from $e_j^{(n)}$ and stopped at the continuous modification $\tau^{(n)}$ of the exit time of $U$.

\begin{thm}[WST boundary-to-boundary branches converge to local multiple $\SLE(2)$]
\label{thm: main result}
In the setup and notation above, $W^{(n)}_\cdot$ converge weakly to the SLE type driving function~\eqref{eq: defn of ptt fcn SLE} stopped at $\tau$, with parameter $\kappa = 2$, and partition function $\PartF_N$ as given in Equation~\eqref{eq: cont ptt fcns - unconditional}. If the WST boundary-to-boundary branches are in addition conditioned to form a given link pattern $\alpha \in \LP_N$, then the analogous convergence holds with the partition function $\PartF_\alpha$ given in~\eqref{eq: cont ptt fcns - conditional}.
\end{thm}

The functions $\PartF_N$ and $\PartF_\alpha$ above are so-called local multiple SLE partition functions at $\kappa=2$, and $W$ is hence a local multiple $\SLE(2)$ driving function; see~\cite[Theorem~4.1]{KKP} or Theorem~\ref{cor: CFT degeneracy PDEs} below.

For several remarkable consequences of Theorem~\ref{thm: main result}, see~\cite[Theorems~5.2,~5.8,~6.8, and Proposition~5.9]{mie2}. Note that for these consequences it is important that no boundary regularity assumptions were imposed on the domains $\domain_n$ or $\domain$.
Some simpler consequences of Theorem~\ref{thm: main result} and by-products of its proof are discussed below.

First, Theorem~\ref{thm: main result} is also a conformal invariance result. Indeed, the description of the scaling limit is given merely in terms of $\bH$ and $X^{(1)}_0  \ldots  X^{(2N)}_0$, not their conformal (pre)images, the actual limiting domain $\domain $ and prime ends $p_1, \ldots, p_{2N}$.

In the case of a single curve, $N=1$, the theorem above is equivalent to the well known $\SLE(2)$ convergence of a WST branch~\cite{Zhan-scaling_limits_of_planar_LERW}. The scaling limit appears as a partition function SLE since the conformal maps were chosen so that it is an $\SLE(2)$ from $X_1$ to $X_2$, not from $0$ to $\infty$.

As a by-product of the proof, we obtain the convergence of the WST boundary-to-boundary branch link pattern probabilities on isoradial graphs and without any boundary regularity assumptions (cf.~\cite{KW-boundary_partitions_in_trees_and_dimers, KW-double_dimer_pairings_and_skew_Young_diagrams} and~\cite[Theorem~3.16]{KKP}).

\begin{thm}[Scaling limit of link pattern probabilities]
\label{thm: scaling limit WST connectivity probas}
In the setup and notation above, the probability that the WST boundary-to-boundary branches form the link pattern $\alpha$ tends to \linebreak $\PartF_\alpha (X^{(1)}_0  \ldots  X^{(2N)}_0) / \PartF_N (X^{(1)}_0  \ldots  X^{(2N)}_0) $ as $n \to \infty$.
\end{thm}

Second, the proof of Theorem~\ref{thm: main result} also provides an alternative proof showing that the partition functions $\PartF_N$ and $\PartF_\alpha$ satisfy the PDEs that appear in the definition of the so-called local multiple SLE partition functions (see, e.g.,~\cite[Appendix~A]{KP-pure_partition_functions_of_multiple_SLEs}). The same PDEs appear in Conformal field theory as degeneracy PDEs for correlation functions of primary fields~\cite{Peltola-towards_a_CFT_for_SLEs}. A different proof for the theorem below was given in~\cite[Theorem~4.1]{KKP} by a direct computation based on the explicit expressions~\eqref{eq: cont ptt fcns - unconditional} and~\eqref{eq: cont ptt fcns - conditional}.

\begin{thm}
\label{cor: CFT degeneracy PDEs}
The partition functions $\PartF_N$ and $\PartF_\alpha$ satisfy for all $j \in \{ 1, \ldots, 2N \}$ the PDEs
\begin{align*}
\partial_{jj}  \PartF_\star (x_1, \ldots, x_{2N}) 
+
\sum_{\substack{i=1 \\ i \neq j }}^{2N}     \frac{ 2 }{ x_i - x_j  } \partial_{i}  \PartF_\star (x_1, \ldots, x_{2N}) 
-
\sum_{\substack{i=1 \\ i \neq j }}^{2N} 
\frac{ 2 }{ (  x_i - x_j  )^2 }  \PartF_\star (x_1, \ldots, x_{2N})   &= 0.
\end{align*}
\end{thm}

\bigskip

\section{The combinatorial model}
\label{sec: connectivity probas}

In this section, we study the combinatorial WST model on a finite connected planar graph $\Gr$. We assume that such $\Gr$ comes with a planar embedding, a choice of boundary vertices, and edge weights $\EdgeWeight$. The main results --- and the only ones referred to in other sections --- are Proposition~\ref{prop: discrete mgales collected}, establishing WST martingales, and Theorem~\ref{thm: disc connectivity probas}, expressing them in terms of discrete harmonic functions.

\subsection{WST connectivity partition functions} 

In this subsection, we define some basic discrete harmonic objects, define the connectivity partition functions of the WST, and review their solution in terms of the discrete harmonic objects, given in~\cite{KW-boundary_partitions_in_trees_and_dimers, KW-double_dimer_pairings_and_skew_Young_diagrams} and~\cite{KKP}.

\subsubsection{The discrete Green's function, Poisson kernel, and excursion kernel}
\label{subsubsec: discrete kernels}

Define the weight $\EdgeWeight (v) $ of a vertex $v$ as the total weight of adjacent edges,
\begin{align*}
\EdgeWeight (v) := \sum_{e = \edgeof{v}{u} \in \Edg} \EdgeWeight (e).
\end{align*}
For $w, v \in \Vert$, denote by $\Walks (v, w)$ the set of finite length nearest-neighbour walks (sequences of adjacent vertices) on the graph $\Gr$, whose first vertex is $v$ and last $w$. Let $\lambda \in \Walks (v, w)$ be a walk with the vertex sequence $v=v_0, v_1, \ldots, v_m = w$ and the edge sequence $e_1 = \edgeof{v_0}{v_1}, e_2 = \edgeof{v_1}{v_2}, \ldots, e_m = \edgeof{v_{m-1}}{v_m}$. Define the weight of the walk $\lambda$ by
\begin{align*}
\EdgeWeight (\lambda) := \frac{ \prod_{k=1}^m \EdgeWeight (e_k) }{ \prod_{k=0}^m \EdgeWeight (v_k) }.
\end{align*}
Note that this weight is preserved under reversing the walk.

Denote by $\Walks^\circ (v, w) \subset \Walks (v, w)$ the walks that only contain interior vertices. The discrete Green's function on $\Gr: \Vert \times \Vert \to \R$ is now defined as the partition function of walks in $\Walks^\circ (v, w)$
\begin{align*}
\GreenK (v, w) := \sum_{\lambda \in \Walks^\circ (v, w) } \EdgeWeight (\lambda).
\end{align*}
Note that $\GreenK (v, w)$ indeed is the Green's function of the negative discrete Laplacian (see Section~\ref{subsec: basic notions of disc harm fcns} in Appendix~\ref{app: disc harm}), and that $\GreenK (v, w)= \GreenK (w, v)$.

If $w$ is the interior vertex of a boundary edge $e \in \bdry \Edg$, then we call $\GreenK (v, w) $ a discrete Poisson kernel between $v$ and $e$ and denote
\begin{align*}
\GreenK (v, w) =: \PoissonK (v, e).
\end{align*}
Note that Green's function has zero boundary values, $\GreenK (v, w) = 0$ for all $w \in \bdry \Vert $, so the Poisson kernel $\PoissonK (v, e)$ can be seen as the discrete derivative of $\GreenK (v, w)$ in $w$, along the boundary edge $e$. Note also that $\PoissonK (v, e)$ is a discrete harmonic function in $v$, see again Section~\ref{subsec: basic notions of disc harm fcns}.

If also $v$ is the interior vertex of a boundary edge $\tilde{e} \in \bdry \Edg$, then we call $\GreenK (v, w) $ a discrete excursion kernel between $\tilde{e}$ and $ e$ and denote
\begin{align*}
\GreenK (v, w) = \PoissonK (v, e) =: \ExcK (\tilde{e}, e).
\end{align*}
The excursion kernel can be interpreted as the discrete derivative of $\PoissonK (v, e)$ in $v$, along the boundary edge $\tilde{e}$.
The reason for introducing these redundant notations is their different behaviour in the scaling limit.

\subsubsection{Excursion kernel determinants}
\label{subsubsec: disc ExcKDet}

Let $\alpha \in \LP_N$ be a link pattern. The \emph{left-to-right orientation} of $\alpha $ is the ordered collection of ordered pairs $((a_1, b_1), \ldots, (a_N, b_N))$ such that $\alpha = \{\{a_1, b_1\}, \ldots, \{a_N, b_N\}\} $ and furthermore $a_i < b_i$ for all $i$ and $a_1 < a_2 < \ldots < a_N$.

Let $e_1, \ldots, e_{2N}$ be boundary edges of $\Gr$, and let $((a_1, b_1), \ldots, (a_N, b_N))$ be the left-to-right orientation of a link pattern $\alpha \in \LP_N$. We define the \emph{excursion kernel determinant} $\LPdet{\alpha}{\ExcK} (e_1 , \ldots , e_{2N})$ of $\alpha$ on $(\Gr; e_1, \ldots, e_{2N})$ by
\begin{align}\label{eq: definition of LPdet in the discrete}
\LPdet{\alpha}{\ExcK} (e_1 , \ldots , e_{2N})
    := \det \Big( \ExcK(e_{a_k}, e_{b_\ell}) \Big)_{k,\ell=1}^N .
\end{align}

\subsubsection{Solution of the WST connectivity partition functions}
\label{subsubsec: soln of connectivity probas}

Consider now the WST measure $\PR$ on $\Gr$. Let $e_1, \ldots, e_{2N}$ be distinct boundary edges, indexed in counterclockwise order. Denote by $E_N$ the event that the WST boundary branches from the interior vertices of the odd edges $e_1, e_3, \ldots, e_{2N - 1}$ reach the boundary $\bdry \Vert$ via the even edges $e_2, e_4, \ldots, e_{2N}$, each using a different even edge. (Recall that this is the event required in the construction of boundary-to-boundary branches in Section~\ref{subsec: combinatorial models}.) We denote
\begin{align*}
Z^{\Gr}_N (e_1 , \ldots , e_{2N}) := \PR [E_N ].
\end{align*}
We wish to keep the graph $\Gr$ (equipped with embedding, boundary vertices, and edge weights), as well as the marked boundary edges, explicit in this notation for later purposes.

For a link pattern $\alpha \in \LP_N$ denote by $E_\alpha$ the event that $E_N$ occurs and additionally the obtained WST boundary-to-boundary branches pair the edges $e_1, \ldots, e_{2N}$ according to the link pattern $\alpha$. Denote
\begin{align*}
Z^{\Gr}_\alpha (e_1 , \ldots , e_{2N}) := \PR [E_\alpha ],
\end{align*}
so obviously
\begin{align}
\label{eq: total disc part fcn}
Z^{\Gr}_N (e_1 , \ldots , e_{2N})  = \sum_{\alpha \in \LP_N} Z^{\Gr}_\alpha (e_1 , \ldots , e_{2N}).
\end{align}
We call $Z^{\Gr}_\alpha$ the \emph{connectivity $\alpha$ partition function of the WST} and $Z^{\Gr}_N$ the \emph{total WST connectivity partition function}.

The WST connectivity partition functions $Z^{\Gr}_\alpha$, for all $\alpha \in \LP_N$, were solved in terms  of excursion kernels determinants in~\cite{KW-boundary_partitions_in_trees_and_dimers, KW-double_dimer_pairings_and_skew_Young_diagrams}. We follow here the presentation in~\cite[Theorem~3.12 and Section~3.6]{KKP}. 

\begin{thm}
\label{thm: disc connectivity probas}
We have for all $\alpha \in \LP_N$
\begin{align}
\label{eq: soln of disc ptt fcns}
Z^{\Gr}_\alpha (e_1 , \ldots , e_{2N}) = \left( \prod_{ \substack{i =2 \\ i \; \mathrm{even}}}^{2N} \EdgeWeight (e_i) \right) \sum_{\beta \in \LP_N} \Mmat^{-1}_{\alpha, \beta}  \LPdet{\beta}{\ExcK} (e_1 , \ldots , e_{2N}),
\end{align}
where $\Mmat^{-1}_{\alpha, \beta} $ only depend on the link patterns $\alpha, \beta \in \LP_N$ as given explicitly in~\cite[Example~2.10]{KKP}.
\end{thm}

The formula above differs from that appearing in~\cite[Section~3.6]{KKP} in terms of the edge weight factor. This is due to a different choice of normalization in the definition of discrete excursion kernels.

\subsection{Discrete martingales}

In this subsection, we study discrete martingales under growing WST boundary-to-boundary branches. The reader should notice that the discussion of this subsection could be carried out more or less similarly in several other lattice models. What is special about the WST is Theorem~\ref{thm: disc connectivity probas} above, which connects the obtained martingales to discrete harmonic functions.

Let us introduce some notation.
Let $(\Gr; e_1, \ldots, e_{2N})$ be as above. Fix a link pattern $\alpha \in \LP_N$.
Denote the conditional WST measures by
\begin{align*}
\PR_N [ \; \cdot \;  ] := \PR [ \; \cdot \; \vert \; E_N ] \qquad \text{and} \qquad \PR_\alpha [ \; \cdot \; ] := \PR [\; \cdot \; \vert \; E_\alpha ],
\end{align*}
where $E_N$ and $E_\alpha \subset E_N$ are as above. 
 Under these conditional measures, we are interested in the WST boundary-to-boundary branch $\gamma$ (a sequence of adjacent vertices $\gamma(0), \gamma(1), \ldots$) starting from the marked boundary edge $e_j = \edgeof{ \gamma(0) }{ \gamma(1) }$. Note that under $\PR_\alpha$, we also know the last edges $ e_k \in \{ e_1, \ldots, e_{2N} \}$ of $\gamma$. Denote by $E_1$ the WST event that the boundary branch from the interior vertex of the odd-index edge $e_j$ or $ e_k$ reaches the boundary $\bdry \Vert$ via the even-index one. Hence $E_\alpha \subset E_1$ and if $N=1$ then indeed $E_N = E_1$. Denote
 \begin{align*}
\PR_1 [ \; \cdot \; ] := \PR [ \; \cdot \; \vert \; E_1 ].
 \end{align*}
Under the conditional measures $\PR_N$, $\PR_\alpha$, and $\PR_1$, denote by $\mathcal{F}_t$, $t \in \{1, 2, \ldots \}$, the sigma algebras of the $t$ first vertices $( \gamma(0), \ldots, \gamma(t) )$ on the path $\gamma$ (so $\mathcal{F}_1$ is the trivial sigma algebra). By discrete martingales we mean martingales under these measures and this filtration.

\subsubsection{Connectivity probability martingales}

Our first martingales are the conditional probabilities of the event $E_\alpha$ given $\mathcal{F}_t$.
Denote by $\Gr_t$ the planar graph $\Gr=(\Vert, \Edg)$ with boundary vertices $\Vert_t^\bdry = \Vert^\bdry \cup \{ \gamma(0), \ldots, \gamma(t-1) \} $ and denote $e_j^{(t)} = \edgeof{\gamma(t-1)}{\gamma(t)}$ (so $\Vert^\bdry_{t=1} = \Vert^\bdry$ and $e_j^{(1)} = e_j$). Define the shorthand notations
\begin{align*}
Z^{ \Gr_t }_\alpha &:= Z^{ \Gr_t }_\alpha (e_1, \ldots, e_{j-1}, e_j^{(t)} , e_{j+1}, \ldots, e_{2N}) \qquad \text{and} \\
Z^{ \Gr_t }_N &:= Z^{ \Gr_t }_N (e_1, \ldots, e_{j-1}, e_j^{(t)} , e_{j+1}, \ldots, e_{2N})  \qquad \text{and} \\
Z^{ \Gr_t }_1 &:=
\begin{cases} Z^{ \Gr_t }_1 ( e_j^{(t)} , e_k), \qquad j \text{ odd} \\
Z^{ \Gr_t }_1 (e_k, e_j^{(t)}  ), \qquad j \text{ even}.
\end{cases}
\end{align*}
Below we use the discrete domain Markov property to construct conditional probability martingales from these partition functions.
Note that Theorem~\ref{thm: disc connectivity probas} expresses these partition functions as polynomials in discrete excursion kernels on $\Gr_t$. 

\begin{lem}
\label{lem: Z-ratios are cond probas}
We have
\begin{align*}
\EX_N [ \mathbbm{1} \{ E_\alpha \} \; \vert \; \mathcal{F}_t ] = Z^{ \Gr_t }_\alpha / Z^{ \Gr_t }_N 
\qquad \text{and} \qquad
\EX_1 [ \mathbbm{1} \{ E_\alpha \}  \; \vert \; \mathcal{F}_t ] = Z^{ \Gr_t }_\alpha / Z^{ \Gr_t }_1.
\end{align*}
\end{lem}

\begin{proof}
Note that we have
\begin{align*}
\EX_N [ \mathbbm{1} \{ E_\alpha \} ]   &= \PR [ E_\alpha  \; \vert \; E_N] 
= \PR [ E_\alpha ] / \PR [ E_N] 
= Z^{ \Gr_1 }_\alpha / Z^{ \Gr_1 }_N ,
\end{align*}
and similarly $\EX_1 [ \mathbbm{1} \{ E_\alpha \} ] = Z^{ \Gr_1 }_\alpha / Z^{ \Gr_1 }_1.$ This actually proves the claim for conditioning on the trivial sigma algebra $\mathcal{F}_1$. The same deduction, combined with the domain Markov property for the WST (see, e.g.,~\cite[Proof of precompactness in Theorem~6.8]{mie2}) can be used to prove the claim for any $\mathcal{F}_t$.
\end{proof}

\subsubsection{Discrete Girsanov transforms}

The measure $\PR_\alpha$ can be seen as either $\PR_N$ or $\PR_1$ conditional on the event $E_\alpha$. We now recall how martingales under the unconditional measure can be transformed to the conditional one and vice versa, by a discrete analogue of Girsanov's transform. Analogous martingale transforms hold in a wide generality but, as with the previous martingales, we prefer to state and prove them for WST, in the form in which they will be applied.

\begin{lem}
\label{lem: discrete Girsanov transforms}
If $M_t^{(\alpha)}$ is an $\mathcal{F}_t$ martingale under $\PR_\alpha$, then
\begin{align*}
M_t^{(N)} = M_t^{(\alpha)} Z^{ \Gr_t }_\alpha / Z^{ \Gr_t }_N 
\end{align*}
is an $\mathcal{F}_t$ martingale under $\PR_N$. If $M_t^{(1)}$ is an $\mathcal{F}_t$ martingale under $\PR_1$, then
\begin{align*}
M_t^{(\alpha)} = M_{t \wedge T}^{(1)} Z^{ \Gr_{t \wedge T} }_1  / Z^{ \Gr_{t \wedge T} }_\alpha
\end{align*}
is an $\mathcal{F}_t$ martingale under $\PR_\alpha$; here $T$ is the $\mathcal{F}_t$ stopping time given by the first time $s$ for which the next step $\gamma(s + 1)$ may be taken under $\PR_1$ so that $Z^{ \Gr_{s + 1} }_\alpha = 0$.
\end{lem}

\begin{proof}
Start from the first transform. The process $M_t^{(N)}$ is clearly adapted, and it is integrable due to the finiteness of the sample space, so it remains to check the conditional expectation property, i.e., $ \EX [ \mathbbm{1}\{ A \}  M_{t+1}^{(N)} ] =  \EX [ \mathbbm{1}\{ A \}  M_{t}^{(N)} ]$ for any $A \in \mathcal{F}_t$. 
Starting from the conditional expectation property of $M_t^{(\alpha)}$, and then expressing $\EX_\alpha$ as a conditional measure, $\EX_\alpha [ \; \cdot \; ] =  \EX_N [\; \cdot \; \mathbbm{1}\{ E_\alpha \} ] /\PR_N [ E_\alpha  ] $, we obtain, for any event $A \in \mathcal{F}_t$
\begin{align*}
\EX_\alpha [ M_t^{(\alpha)} \mathbbm{1}\{ A \} ] &= \EX_\alpha [ M_{t+1}^{(\alpha)} \mathbbm{1}\{ A \} ]  \\
\EX_N [ M_t^{(\alpha)}  \mathbbm{1}\{ A \} \EX_N [ \mathbbm{1} \{ E_\alpha \}  \; \vert \; \mathcal{F}_t ] ] &= \EX_N [ M_{t+1}^{(\alpha)} \mathbbm{1}\{ A \} \EX_N [ \mathbbm{1} \{ E_\alpha \} \; \vert \; \mathcal{F}_{t+1} ]  ].
\end{align*}
Substituting the conditional probabilities from Lemma~\ref{lem: Z-ratios are cond probas} now proves the first claim.

For the second transform, integrability and adaptedness are similarly clear. Let us prove the conditional expectation property. Notice that 
$
\EX_\alpha [ \; \cdot \; ] = \EX_1 [\; \cdot \; \mathbbm{1}\{ E_\alpha \} ] / \PR_1 [ E_\alpha  ] 
$
and using  Lemma~\ref{lem: Z-ratios are cond probas} compute, for an arbitrary event $A \in \mathcal{F}_t$
\begin{align}
\label{eq: test exp difference}
\EX_\alpha & [ M_t^{(\alpha)} \mathbbm{1}\{ A \} ] - \EX_\alpha [ M_{t+1}^{(\alpha)} \mathbbm{1}\{ A \} ] \\
\nonumber
&= \frac{1}{ \PR_1 [ E_\alpha  ] } \left(
\EX_1 \left[ M_t^{(\alpha)} \mathbbm{1}\{ A \}   Z^{ \Gr_{t } }_\alpha  / Z^{ \Gr_{t } }_1 \right]
-
\EX_1 \left[ M_{t+1}^{(\alpha)} \mathbbm{1}\{ A \} Z^{ \Gr_{t +1} }_\alpha  / Z^{ \Gr_{t + 1 } }_1 \right] 
\right).
\end{align}
Next, using the piecewise definition of $M_{t}^{(\alpha)} $, we obtain
\begin{align*}
\eqref{eq: test exp difference} 
=
& \frac{1}{ \PR_1 [ E_\alpha  ] } \Bigg(
\EX_1 \bigg[ \underbrace{ \mathbbm{1}\{ T \le t \}  M_{T}^{(1)} ( Z^{ \Gr_{T} }_1  / Z^{ \Gr_{ T} }_\alpha ) \mathbbm{1}\{ A \} }_{ := X_t, \quad \mathcal{F}_t \text{-measurable r.v.}} \big(  \underbrace{  Z^{ \Gr_{t } }_\alpha  / Z^{ \Gr_{t } }_1  -  Z^{ \Gr_{t +1} }_\alpha  / Z^{ \Gr_{t + 1 } }_1 }_{ \text{apply Lemma }  \ref{lem: Z-ratios are cond probas} } \big) \bigg] \\
&
\qquad \qquad
+
\EX_1 \bigg[ \underbrace{ \mathbbm{1}\{ T > t \} \mathbbm{1}\{ A \} }_{ := Y_t, \quad \mathcal{F}_t \text{-measurable r.v.}}   M_{t}^{(1)} ( Z^{ \Gr_{t} }_1  / Z^{ \Gr_{t} }_\alpha )    Z^{ \Gr_{t } }_\alpha  / Z^{ \Gr_{t } }_1  \bigg]
\\
&
\qquad \qquad
-
\EX_1 \bigg[ \underbrace{ \mathbbm{1}\{ T > t \} \mathbbm{1}\{ A \} }_{ = Y_t }   M_{t + 1}^{(1)} ( Z^{ \Gr_{t + 1} }_1  / Z^{ \Gr_{t + 1} }_\alpha )  Z^{ \Gr_{t +1} }_\alpha  / Z^{ \Gr_{t + 1 } }_1 \bigg] 
\Bigg) 
\\
=& 
\frac{1}{ \PR_1 [ E_\alpha  ] } \Bigg(
\underbrace{ \EX_1 \bigg[  X_t \EX_1 [ \mathbbm{1} \{ E_\alpha \}  \; \vert \; \mathcal{F}_t]  \bigg]  - \EX_1 \bigg[ X_t \EX_1 [ \mathbbm{1} \{ E_\alpha \}  \; \vert \; \mathcal{F}_{t+1}] }_{ = 0 \text{ by tower law of conditional expectation}} \bigg]
+
\underbrace{ \EX_1 \left[ Y_t M_{t}^{(1)}    \right]
-
\EX_1 \left[ Y_t  M_{t + 1}^{(1)}  \right] }_{= 0 \text{ since $M_{t}^{(1)} $ is a martingale}}
\Bigg) \\
= &
 0.
\end{align*}
This finishes the proof.
\end{proof}

\subsubsection{Discrete harmonic martingales}

We now establish the discrete harmonic martingales on which the scaling limit identification is based. The starting point is the well-known $\mathcal{F}_t$ martingales under the measure $\PR_1$, given by~\cite{LSW-LERW_and_UST, Zhan-scaling_limits_of_planar_LERW, CW-mLERW}
\begin{align}
\label{eq: classical LERW mgale}
M_t(v ) &= \frac{\PoissonK^{\Gr_t} (v, e_j^{(t)} )  }{ \ExcK^{\Gr_t } ( e_k, e_j^{(t)} ) }, 
\qquad
\text{where $v \in \Vert$ is any fixed vertex;}
\end{align}
here and hereafter we will need Poisson and excursion kernels on subgraphs $\Gr_t$ of $\Gr$, explicating the subgraph in the superscript.
Note that for fixed $t$, $M_t(v)$ is a discrete harmonic function of $v$ on $\Gr_t$ (see Section~\ref{subsec: basic notions of disc harm fcns} in Appendix~\ref{app: disc harm}).
Let us also define the notation
\begin{align*}
\tilde{Z}_\alpha^{\Gr_t} := \sum_{\beta \in \LP_N} \Mmat^{-1}_{\alpha, \beta}  \LPdet{\beta}{\ExcK^{\Gr_t}} (e_1 , \ldots , e_{j-1}, e_j^{(t)}, e_{j+1}, \ldots,  e_{2N}),
\end{align*}
i.e., $\tilde{Z}_\alpha^{\Gr_t} $ is obtained from ${Z}_\alpha^{\Gr_t} $ given by~\eqref{eq: soln of disc ptt fcns} by dividing out the weights of the even marked boundary edges of $\Gr_t$. Analogously, define $\tilde{Z}^{ \Gr_{t } }_N = \sum_{\beta \in \LP_N} \tilde{Z}^{ \Gr_{t } }_\beta$ and $\tilde{Z}^{ \Gr_{t } }_1$ as the case $N=1$ with marked edges $e_j^{(t)}, e_k$. 
The martingale transforms of Lemma~\ref{lem: discrete Girsanov transforms} and the martingale~\eqref{eq: classical LERW mgale} now allow us to find martingales under the various measures. We collect these below, directly with normalizing factors for which scaling limits exist. Note that the normalizing factors are Poisson kernels $\PoissonK = \PoissonK^{\Gr_1}$ at time $t=1$.

\begin{proposition}
\label{prop: discrete mgales collected}
Let $v, w \in \Vert$ be any vertices and $\beta \in \LP_N$ any link pattern. We have the following $\mathcal{F}_t$ martingales under the different measures:
\begin{align*}
M^{(1)}_t(v, w ) &= \frac{\PoissonK^{\Gr_t} (v, e_j^{(t)} )  }{  \tilde{Z}^{ \Gr_{t } }_1 } \PoissonK (w, e_k ) , \qquad \text{under } \PR_1; \\
M^{(\alpha)}_t(v, w ) &= \frac{\PoissonK^{\Gr_{t } } (v, e_j^{(t)} )  }{ \tilde{Z}^{ \Gr_{t } }_\alpha } \prod_{ \substack{ i = 1 \\ i \neq j} }^{2N} \PoissonK (w, e_i )  \qquad \text{stopped at $T$, under } \PR_\alpha; \\
M^{(N)}_t(v, w ) &= \frac{\PoissonK^{\Gr_{t } } (v, e_j^{(t)} )  }{ \tilde{Z}^{ \Gr_{t } }_N } \prod_{ \substack{ i = 1 \\ i \neq j} }^{2N} \PoissonK (w, e_i )  \qquad \text{stopped at $T$, under }  \PR_N; \\
\tilde{M}^{(N)}_t &= \frac{ \tilde{Z}^{ \Gr_{t} }_\beta  }{ \tilde{Z}^{ \Gr_{t } }_N } \qquad \text{under } \PR_N; \text{ and} \\
\tilde{M}^{(\alpha)}_t &= \frac{ \tilde{Z}^{ \Gr_{t} }_\beta  }{ \tilde{Z}^{ \Gr_{t } }_\alpha } \qquad \text{stopped at $T_N$, under } \PR_\alpha; \\
\end{align*}
here $T$ (resp. $T_N$) is the $\mathcal{F}_t$ stopping time given by the first time $s$ for which the next step $\gamma(s + 1)$ may be taken under $\PR_1$ (resp. $\PR_N$) so that $Z^{ \Gr_{s + 1} }_\alpha = 0$.
\end{proposition}

\begin{proof}
For the first martingale, note that $\tilde{Z}^{ \Gr_{t } }_1 $ is by Theorem~\ref{thm: disc connectivity probas} a constant scaling of $\ExcK^{\Gr_t } ( e_k, e_j^{(t)} ) $. $M^{(1)}_t(v, w )$ is thus a constant scaling of the martingale~\eqref{eq: classical LERW mgale}. The second one is, up to constant scaling, obtained by applying the second martingale transform of Lemma~\ref{lem: discrete Girsanov transforms} to the first martingale.
The third martingale $M^{(N)}_t(v, w )$ is obtained by the first transform of Lemma~\ref{lem: discrete Girsanov transforms} from the second martingale $M^{(\alpha)}_t(v, w )$.
The fourth one is the conditional probability martingale of Lemma~ \ref{lem: Z-ratios are cond probas}.
The fifth martingale $\tilde{M}^{(\alpha)}_t$ is obtained by the second martingale transform of Lemma~\ref{lem: discrete Girsanov transforms} from the fourth one. (Lemma~\ref{lem: discrete Girsanov transforms} is stated as transforming $\PR_1$ martingales to $\PR_\alpha$ but its direct analogue applies from $\PR_N$  to $\PR_\alpha$.)
\end{proof}

\bigskip 

\section{Observable convergence results}
\label{sec: obs conv}

\subsection{Convergence of discrete harmonic objects}

\subsubsection{The continuous Green's function, Poisson kernel, and excursion kernel}
\label{subsubsec: cts kernels}

The Green's function $\GreenKH (z, w)$ of the negative Laplacian $(- \Delta)$ on the upper half-plane $z, w \in \bH$ is given by
\begin{align*}
\GreenKH (z, w) = - \frac{1}{2 \pi} [\log \vert z - w \vert - \log \vert z - w^* \vert].
\end{align*}
The \emph{Poisson kernel} $\PoissonKH (z, x)$ of the $z \in \bH$ at $x \in \R$ and is given by
\begin{align}
\label{eq: defn of cont poisson kernel}
\PoissonKH (z, x) = - \frac{1}{\pi} \Im (\frac{1}{z-x}) = \frac{1}{\pi} \frac{\Im (z)}{\vert z-x \vert^2},
\end{align}
and the \emph{excursion kernel} $\ExcKH(x, y)$ between $x, y \in \R$, $x \ne y$ is given by
\begin{equation*}
\ExcKH(x, y) = \frac{1}{\pi} \frac{1}{(x-y)^2}.
\end{equation*}

Notice that the normal derivatives of $\GreenKH$ can be defined on $\R$ by Schwarz reflection, and then
\begin{align*}
\PoissonKH (z, x) = \left( \partial_y \GreenKH (z, x + yi) \right)_{y=0},
\end{align*}
and in a similar sense
\begin{align*}
\ExcKH(x, x') = \left( \partial_y \PoissonKH (x + yi, x') \right)_{y=0}.
\end{align*}
With this observation, $\GreenKH$, $\PoissonKH$, and $\ExcKH$ are the continuum analogues (in $\bH$) of the discrete objects $\GreenK$, $\PoissonK$, and $\ExcK$, respectively, introduced in Section~\ref{subsubsec: discrete kernels}.

When combining Poisson kernels with Loewner evolutions, we will need the following elementary observation. Let $H_t \subset \bH$ be the complement of a hull in $\bH$ and $g_t: H_t \to \bH$ its mapping-out function. Let $x \in \R$ be such that a neighbourhood $N$ of $x$ in $\bH$ is contained in $H_t$. As observed above, $\GreenKH (z, w)$ defines a harmonic function of $w$ on $\bH \setminus \{ z \}$ (anf thus on $H_t \setminus \{ z \}$) with boundary normal derivative $\PoissonKH (z, x)$ at $x$. Denote $z_t = g_t (z)$, $w_t = g_t (w)$, and $x_t = g_t (x)$. By conformal invariance of harmonic functions, $\GreenKH (z, w) = \GreenKH (g_t^{-1} (z_t), g_t^{-1} (w_t))$ is hence a harmonic function of $w_t$ on $\bH \setminus \{ z_t \}$. Its boundary normal derivative at $x_t$ given by
\begin{align*}
(g_t^{-1} )'(x_t) \PoissonKH (z, x) = \frac{ 1 }{ g_t'(x) } \PoissonKH (z, x).
\end{align*}

\subsubsection{Convergence results on isoradial graphs}

We now state four results, guaranteeing the convergence of suitable ratios of discrete Green's functions, Poisson kernels, and excursion kernels to their continuous counterparts. The two first results are ``classics'' of discrete harmonic analysis, see~\cite{CS-discrete_complex_analysis_on_isoradial}, while the two latter ones follow by combining the first ones with Theorem~\ref{thm: ratios of discrete harmonic functions converge} of Appendix~\ref{app: disc harm}. Theorem~\ref{thm: ratios of discrete harmonic functions converge} can be proven based on a uniform estimate on the behaviour discrete harmonic functions near a boundary segment with zero boundary conditions, given recently by Chelkak and Wan~\cite[Corollary~3.8]{CW-mLERW}.
We provide in Appendix~\ref{app: disc harm} a different proof based on conformal crossing estimates for the random walk~\cite{KS}, that was found independently by the author.

The convergence results consider the following setup. Let $\Gr^{(n)} = (\Vert^{(n)}, \Edg^{(n)})$ be simply-connected subgraphs of the isoradial lattices $\InfiniteGr^{(n)}$ with mesh sizes $\delta_n \to 0$, as defined in Section~\ref{subsec: isoradial graphs}. Denote the Green, Poisson, and excursion kernels on $\Gr^{(n)}$ by $\GreenK^{(n)}$, $\PoissonK^{(n)}$, and $\ExcK^{(n)}$, respectively.
 Let $v^{(n)}, w^{(n)} \in \Vert^{(n)}$ be interior vertices and $e_1^{(n)}, e_2^{(n)} \in \bdry \Edg^{(n)}$ be distinct boundary edges, both connected to $v^{(n)}$ by a path on the interior vertices. Assume that $(\Gr^{(n)}; v^{(n)}, w^{(n)}; e_1^{(n)}, e_2^{(n)}) \to (\domain; v, w; p_1, p_2)$ in the Carath\'{e}odory sense, where the limit is a simply-connected domain with two marked interior points and two distinct marked prime ends. Let $\confmap$ be a conformal map $\domain \to \bH$. Note that the scaling limits in the following theorem are conformally invariant, in the sense that they do not depend on the precise choice of this conformal map.

\begin{thm}
\label{prop: disc harm conv results}
In the setup and notation given above, we have the following convergences as $n \to \infty$.
\begin{itemize}
\item[i)] 
\cite[Corollary~3.11]{CS-discrete_complex_analysis_on_isoradial}
The discrete Green's functions $\GreenK^{(n)} (\cdot, v^{(n)})$ tend to the continuous one $\GreenKH ( \confmap (\cdot), \confmap (v) ) $ uniformly over compact subsets of $\domain \setminus \{ v \}$, in the following precise sense:
given $r > 0$, there exist $\eps(n) = \eps(n, r) $ with $\eps(n) \to 0 $ as $n \to \infty$, such that for all vertices $u \in \Vert^{(n)}$ lying inside the limiting domain $\domain$ with $d(u, v), d (u, \bdry \domain )\geq r$, we have
\begin{align*}
\big\vert  \GreenK^{(n)} (u, v^{(n)}) - \GreenKH ( \confmap (u), \confmap (v) ) \big\vert \leq \eps(n).
\end{align*}
\item[ii)] 
\cite[Theorem~3.13]{CS-discrete_complex_analysis_on_isoradial} Ratios of discrete Poisson kernels $\PoissonK^{(n)} (\cdot, e^{(n)}_1) / \PoissonK^{(n)}  (v^{(n)}, e^{(n)}_1)$ tend to the continuous ones $\PoissonKH (\confmap(\cdot), \confmap( p_1 )) / \PoissonKH ( \confmap(v) , \confmap(p_1 ))$ uniformly over compact subsets of $\domain $, in the following precise sense:
given $r > 0$, there exist $\eps(n) = \eps(n, r) $ with $\eps(n) \to 0 $ as $n \to \infty$, such that for all vertices $u \in \Vert^{(n)}$ lying inside the limiting domain $\domain$ with $ d (u, \bdry \domain )\geq r$, we have
\begin{align*}
\left\vert  \frac{\PoissonK^{(n)} (u, e^{(n)}_1) }{ \PoissonK^{(n)}  (v^{(n)}, e^{(n)}_1) }
- \frac{ \PoissonKH (\confmap(u ), \confmap( p_1 )) }{ \PoissonKH (\confmap(v), \confmap( p_1 )) }  \right\vert
 \leq \eps(n);
\end{align*}
here we assume that $\confmap$ is chosen so that $\confmap (p_1) \neq \infty$.
\item[iii)] \emph{Convergence of excursion kernel--Poisson kernel ratios (also in \cite[Proposition~3.14]{CW-mLERW}):} we have
\begin{align*}
 \frac{ \ExcK^{(n)} (e^{(n)}_1, e^{(n)}_2) }{ \PoissonK^{(n)} (v^{(n)}, e^{(n)}_1) \PoissonK^{(n)} (w^{(n)}, e^{(n)}_2) } \longrightarrow \frac{ \ExcKH ( \confmap (p_1), \confmap (p_2) ) }{ \PoissonKH ( \confmap (v), \confmap (p_1) ) \PoissonKH ( \confmap (w), \confmap (p_2) ) }, \qquad \text{as } n \to \infty
\end{align*}
where we assume that $\confmap (p_1), \confmap (p_2) \neq \infty$.
\item[iv)] \emph{Convergence of ratios of Poisson kernels in different domains:} 
Let $\tilde{\Gr}^{(n)} \subset \Gr^{(n)}$ be simply-connected subgraphs of $\InfiniteGr^{(n)}$, such that also $(\tilde{\Gr}^{(n)}; v^{(n)}, w^{(n)}; e_1^{(n)} )$ satisfy the assumptions of this proposition, with the limiting domain $(\tilde{\domain}; v, w; p_1)$. Suppose furthermore that $\bH \setminus \confmap ( \tilde{\domain}  )$ is a hull and bounded away from $\confmap (w)$ and $\confmap (p_1)$. Then, we have
\begin{align*}
\frac{ \PoissonK^{ \Gr^{(n)} } (v^{(n)}, e_1^{(n)} )}{ \PoissonK^{ \tilde{\Gr}^{(n)} } (w^{(n)}, e_1^{(n)})} \longrightarrow \frac{ \frac{ 1 }{ g'( \confmap (p_1) ) } \PoissonKH ( \confmap (v) , \confmap (p_1) ) }{ \PoissonKH ( g ( \confmap (w) ), g ( \confmap (p_1) ) ) }  \qquad \text{as } n \to \infty,
\end{align*}
where $g$ is a conformal mapping-out function $\confmap ( \tilde{\domain} ) \to \bH$ and we assume $\confmap (p_1) \neq \infty$.
\end{itemize}
\end{thm}

The proof of Theorem~\ref{prop: disc harm conv results} is given in Section~\ref{subsec: proof of disc harm conv res} in Appendix~\ref{app: disc harm}.

\subsection{SLE partition functions and convergence of WST connectivity probabilities}

\subsubsection{Excursion kernel determinants and partition functions}
\label{subsubsec: cont ExcKDet}

In analogy to the discrete excursion kernel determinants, as defined in Section~\ref{subsubsec: disc ExcKDet}, we define their continuous counterparts.
Let $x_1 < \ldots < x_{2N} $ be real numbers, and let $((a_1, b_1), \ldots, (a_N, b_N))$ be the left-to-right orientation of a link pattern $\alpha \in \LP_N$. We define the \emph{continuous excursion kernel determinant} $\LPdet{\alpha}{\ExcKH} (x_1 , \ldots , x_{2N})$ of $\alpha$ by
\begin{align*}
\LPdet{\alpha}{\ExcKH} (x_1 , \ldots , x_{2N})
    := \det \Big( \ExcKH(x_{a_k}, x_{b_\ell}) \Big)_{k,\ell=1}^N .
\end{align*}
and connectivity partition functions
\begin{align}
\label{eq: cont ptt fcns - conditional}
\PartF_\alpha (x_1 , \ldots , x_{2N}) := \sum_{\beta \in \LP_N} \Mmat^{-1}_{\alpha, \beta}  \LPdet{\beta}{\ExcKH} (x_1 , \ldots , x_{2N}),
\end{align}
where $\Mmat^{-1}_{\alpha, \beta}$ is as in Theorem~\ref{thm: disc connectivity probas}. Finally, we also define
\begin{align}
\label{eq: cont ptt fcns - unconditional}
\PartF_N (x_1 , \ldots , x_{2N}) := \sum_{\alpha \in \LP_N} \PartF_\alpha (x_1 , \ldots , x_{2N}) .
\end{align}
An alternative expression for $\PartF_N$ is given in \cite[Lemma~4.12]{PW}.

\subsubsection{Proof of Theorem~\ref{thm: scaling limit WST connectivity probas}}
Denote by $\PR^{(n)}$ the WST measure on $\Gr^{(n)}$, let $E^{(n)}_N$ and $E^{(n)}_\alpha$ be the WST connectivity events on $(\Gr^{(n)}; e_1^{(n)}, \ldots, e_{2N}^{(n)})$ defined in Section~\ref{subsubsec: soln of connectivity probas}, and $\PR^{(n)}_N [ \; \cdot \;] = \PR^{(n)} [ \; \cdot \; \vert E^{(n)}_N]$. In this notation, we wish to show that 
\begin{align*}
\PR^{(n)}_N [ E^{(n)}_\alpha ] \longrightarrow 
\frac{\PartF_\alpha (X^{(1)}_0  \ldots  X^{(2N)}_0) }{ \PartF_N (X^{(1)}_0  \ldots  X^{(2N)}_0)}
\qquad
\text{as } n \to \infty.
\end{align*}

From the inclusion of events $E^{(n)}_\alpha \subset E^{(n)}_N$, we have 
\begin{align}
\nonumber
\PR^{(n)}_N [ E_\alpha^{(n)} ] &= \frac{ \PR^{(n)} [ E_\alpha^{(n)} ] }{ \PR^{(n)} [ E_N^{(n)} ] } \\
\label{eq: disc conn probas}
\text{(Theorem~\ref{thm: disc connectivity probas})} &= \frac{ \sum_{\beta \in \LP_N} \Mmat^{-1}_{\alpha, \beta}  \LPdet{\beta}{\ExcK^{(n)}} (e^{(n)}_1 , \ldots , e^{(n)}_{2N}) }{ \sum_{\gamma \in \LP_N} \sum_{\beta \in \LP_N} \Mmat^{-1}_{\gamma, \beta}  \LPdet{\beta}{\ExcK^{(n)}} (e^{(n)}_1 , \ldots , e^{(n)}_{2N}) }.
\end{align}
Note that each term in the determinants $\LPdet{\beta}{\ExcK^{(n)}}$ above is a product of $N$ excursion kernels $\ExcK^{(n)}(\cdot, \cdot)$, and in such a term, each of the $2N$ marked boundary edges appears as an excursion kernel argument exactly once.  Divide both sides of the fraction above by $\prod_{i=1}^{2N} \PoissonK^{(n)} (v^{(n)}, e^{(n)}_i) $, where $v^{(n)} \in \Vert_n$ is the vertex of $\Gr^{(n)}$ closest to a fixed but arbitrary reference point $v $ in the limiting domain $ \domain$. Then, by Theorem~\ref{prop: disc harm conv results}(iii), studying either the numerator or denominator of~\eqref{eq: disc conn probas} above (but not yet their ratio), in the scaling limit $n \to \infty$ we can replace the discrete Poisson and excursion kernels by their continuous counterparts, making a small error $o(1)$. That is, for instance for the numerator, we compute
\begin{align}
\nonumber
\prod_{i=1}^{2N} & \left( \frac{1}{ \PoissonK^{(n)} (v^{(n)}, e^{(n)}_i)} \right) \sum_{\beta \in \LP_N} \Mmat^{-1}_{\alpha, \beta}  \LPdet{\beta}{\ExcK^{(n)}} (e^{(n)}_1 , \ldots , e^{(n)}_{2N}) \\
\nonumber
&= \prod_{i=1}^{2N} \left( \frac{1}{ \PoissonKH (\confmap (v ), \confmap ( p_i ) )} \right) \sum_{\beta \in \LP_N} \Mmat^{-1}_{\alpha, \beta}  \LPdet{\beta}{\ExcKH} ( \confmap ( p_1 ) , \ldots , \confmap ( p_{2N} ) )  + o(1) \\
\label{eq: scaling limit PartF}
&= \prod_{i=1}^{2N} \left( \frac{1}{ \PoissonKH (\confmap (v ),  X^{(i)}_0 )} \right)
\PartF_\alpha (X^{(1)}_0  \ldots  X^{(2N)}_0)  + o(1),
\end{align}
where the last step used the definitions~\eqref{eq: cont ptt fcns - conditional} and $ X^{(i)}_0 = \confmap (p_i)  $. 
Furthermore, note that by~\cite[Theorem~4.1]{KKP}, we have $\PartF_\alpha > 0$, so the error $o(1)$ in~\eqref{eq: scaling limit PartF} is small also \textit{relative} to the first term. A similar deduction holds for the denominator of~\eqref{eq: disc conn probas}. Due to small \textit{relative} errors, we can also study the ratio~\eqref{eq: disc conn probas}, and Theorem~\ref{thm: scaling limit WST connectivity probas} follows.
{ \ }$\hfill \qed$


\bigskip

\section{Proof of the main theorem}
\label{sec: proof of main thm}

The proof of Theorem~\ref{thm: main result} consists of showing precompactness, i.e., the existence of subsequential weak limits, and identification of any subsequential limit. The precompactness part was done for one curve in~\cite{KS} (see also~\cite{mie}), and for multiple curves in~\cite{mie2}. We briefly review the key part of the argument in Section~\ref{subsubsec: precompactness} in Appendix~\ref{app: disc harm}. This section provides the proof of the identification part.

For notation, denote by $\PR^{(n)}_\star$ the WST measure on $\Gr^{(n)}$, conditional on the event $E^{(n)}_\star$ between the edges $e^{(n)}_1, \ldots, e^{(n)}_{2N}$, where $\star \in \{ \alpha, N, 1 \}$. (The limit identification will be identical for $\star \in \{ \alpha, N, 1 \}$.) By the precompactness, we may extract a subsequence such that the stopped driving functions $W^{(n)}$ converge weakly to a limiting random function $W$ described by the weak limit measure $\limPR_\star$ (a Borel measure on the space of continuous functions). We will suppress the subsequence notation and assume that $W^{(n)}$ converge weakly.

\subsection{Continuous martingales in the scaling limit}
\label{subsec: cts mgales}

The first step in the identification part of Theorem~\ref{thm: main result} is to promote the discrete martingales of Proposition~\ref{prop: discrete mgales collected} to continuous martingales in the weak limit. This is formulated in Proposition~\ref{prop: cts martingale} below, and the rest of this subsection constitutes the proof of that proposition.
 
To state Proposition~\ref{prop: cts martingale}, notice that the derivative $g_t '$ of a Loewner mapping-out function (see~\eqref{eq: Loewner ODE}) evolves as
\begin{align}
\nonumber
g_0 ' (z) & = 1 \\
\label{eq: LE for derivative}
\partial_t g_t ' (z) & = - \frac{ 2 g_t ' (z) }{ (g_t (z) - W_t )^2}.
\end{align}
Up to the stopping time $\tau$, the functions $g_t(\cdot)$ and their derivatives $g_t'(\cdot)$ are well defined by Schwarz reflection also at the marked boundary points $X^{(i)}_0 \in \R$, $i \neq j$. Their evolution is governed by the same differential equation~\eqref{eq: LE for derivative}, with $z= X^{(i)}_0$. Recall also the definitions of the neighbourhood $U_\epsilon$ and the filtration $\mathscr{F}_t$ from Section~\ref{subsec: LE and SLE}.

\begin{proposition}
\label{prop: cts martingale}
For all $z \in \bH \setminus U_\epsilon$ and $\omega \in \bH$, the process
\begin{align}
\label{eq: cts mgale}
\Mart^{(\star)}_t( z , \omega) &= \frac{ \PoissonKH ( g_t(z) , W_t )  }{ \PartF_\star (X^{(1)}_t,  \ldots,  X^{(2N)}_t) } \prod_{ \substack{ i = 1 \\ i \neq j} }^{2N} \frac{\PoissonKH ( \omega , X^{(i)}_0)}{ g_t' (X^{(i)}_0)}, \qquad \text{stopped at } \tau
\end{align}
is a continuous bounded $\mathscr{F}_t$ martingale under $\limPR_\star$.
\end{proposition}

\subsubsection{Proof of boundedness and continuity in Proposition~\ref{prop: cts martingale}}

Boundedness: The boundedness of the process $\Mart^{(\star)}_t( z , \omega)$ follows from basic properties of Loewner evolutions, combined with some standard harmonic measure arguments. A proof is given for completeness in Appendix~\ref{app: boundedness pf}.

Continuity: Recall that $ W_t$ is the weak limit process on the space of continuous functions, thus by construction continuous in $t$. From basic properties of ordinary differential equations, it follows that also the processes $X^{(1)}_t,  \ldots,  X^{(2N)}_t$, $ g_t(z)$, and $g_t' (X^{(i)}_0)$ are then continuous. Thus, each individual factor in the denominator and numerator of the right-hand side of~\eqref{eq: cts mgale} is continuous. Finally, in the proof of boundedness it is shown that the processes in the denominator remain bounded away from zero. Continuity of $\Mart^{(\star)}_t( z )$ then follows.
{ \ }$\hfill \qed$

\subsubsection{Uniform convergence of discrete martingale observables}
\label{subsubsec: uniform conf of mg obs}

Before proceeding to prove the martingaleness in Proposition~\ref{prop: cts martingale}, we will need a uniform convergence result for the discrete martingale observables in Proposition~\ref{prop: discrete mgales collected}. 

In order to state the uniform convergence, we need some more notations. View the WST boundary-to-boundary branch from $e^{(n)}_j$, as mapped to $\bH$ by $\confmap_n$, as a Loewner chain. Denote by $t$ the continuous time parameter of this Loewner chain and by $\tau^{(n)}$ the continuous exit time of the localization neighbourhood $U$. Denote by $W^{(n)}_\cdot$ the driving function of the Loewner chain, as stopped at $\tau^{(n)}$ (so $W^{(n)}_\cdot \to W_\cdot$ weakly), and by $g^{(n)}_\cdot$ the solutions to the Loewner equation. Denote $X^{(n;i)}_t = g^{(n)}_t (X^{(n;i)}_0)$ the solutions of this Loewner equation starting from the boundary point $X^{(n;i)}_0 = \confmap_n(e^{(n)}_i)$ corresponding to the $i$:th marked boundary edge. For $z \in \bH \setminus U_\epsilon$ and $\omega \in \bH$, denote
\begin{align}
\label{eq: def of disc-cont process}
\Mart^{(n; \star)}_t( z , \omega) = \frac{ \PoissonKH ( g^{(n)}_t(z) , W^{(n)}_t )  }{ \PartF_\star (X^{(n; 1)}_t,  \ldots,  X^{(n; 2N)}_t) } \prod_{ \substack{ i = 1 \\ i \neq j} }^{2N} \frac{\PoissonKH ( \omega , X^{(n; i)}_0 )}{ (g^{(n)}_t)' (X^{(n; i)}_0)}, \qquad \text{stopped at } \tau
\end{align}
Finally, denote by $\lceil t \rceil^{(n)}$ (resp. $\lceil \tau^{(n)} \rceil^{(n)}$) the next time after $t$ (resp. $ \tau^{(n)} $) when the growth process has reached a vertex of $\Vert_n$, as intepreted on a WST branch on $\Gr_n$. 

\begin{proposition} 
\label{prop: unif conv of mgale obs}
Assume the setup of Theorem~\ref{thm: main result}, and let $r > 0$ be given.  There exist $\eps(n) = \eps(n, r)$, with $\eps(n) \to 0$ as $n \to \infty$, such that the following holds. For any $v, w \in \Vert_n$
with $\vert \confmap_n (v) \vert, \vert \confmap_n (w) \vert < 1/r$ and $d( \confmap_n (v) , \bdry ( \bH \setminus U_\epsilon)), d(\confmap_n (w), \bdry \bH) > r$, 
any realization of $W^{(n)}_\cdot$ possible under $\PR^{(n)}_\star$,
and any $t \leq \tau^{(n)}$
\begin{align*}
\left\vert
M^{(n; \star)}_{\lceil t \rceil^{(n)}}(v, w) - \Mart^{(n; \star)}_t( \confmap_n (v) , \confmap_n (w) )
\right\vert
\leq \eps (n),
\end{align*}
and thus in particular
\begin{align*}
\left\vert
M^{(n; \star)}_{\lceil \tau^{(n)} \rceil^{(n)}}(v, w) - \Mart^{(n; \star)}_{\tau^{(n)}} ( \confmap_n (v) , \confmap_n (w) )
\right\vert
\leq \eps (n);
\end{align*}
here $M^{(n; \star)}_{s}(v, w) $ are the $\PR^{(n)}_\star$ martingales in discrete time $s$ from Proposition~\ref{prop: discrete mgales collected}, for the WST boundary-to-boundary branches on $(\Gr_n; e_1^{(n)}, \ldots, e_{2N}^{(n)})$.
\end{proposition}

\begin{proof}
Fix a realization of the WST boundary-to-boundary branch from $e^{(n)}_j$, and the corresponding driving function $W^{(n)}_\cdot$. Denote by $s$ the discrete time parameter, and fix also $s \leq \lceil \tau^{(n)} \rceil^{(n)}$. We have thus also fixed the graph $\Gr^{(n)}_{s }$. Recall the expression for $M^{(n;\star)}_s (v, w )$ from Proposition~\ref{prop: discrete mgales collected} (suppressing all indices $n$ in the expression to streamline the notation):
\begin{align}
\label{eq: disc mg obs expanded}
M^{(n; \star)}_s (v, w ) &= \frac{\PoissonK^{\Gr_{s } } (v, e_j^{(s)} )  }{ \tilde{Z}^{ \Gr_{s } }_\star }  \prod_{ \substack{ i = 1 \\ i \neq j} }^{2N} \PoissonK (w, e_i )
= \frac{  \prod_{  i = 1  }^{2N} \PoissonK^{\Gr_{s } } (v, e_i^{(s)} )  }{ \tilde{Z}^{ \Gr_{s } }_\star }   \prod_{ \substack{ i = 1 \\ i \neq j} }^{2N} \frac{ \PoissonK (w, e_i ) }{\PoissonK^{\Gr_{s } } (v, e_i )}  
\qquad \text{stopped at $s=T$ },
\end{align}
where we denoted $e_i^{(s)} = e_i$ for $i \neq j$.
Notice that (for all $n$ large enough) $\lceil \tau^{(n)} \rceil^{(n)}$
comes before $T$ , so we need not care about the stopping at $T$ in what follows.

Now, assume for a contradiction that for some $\delta > 0$, there existed infinitely many $n$, $W^{(n)}_\cdot$, $t^{(n)} \leq \tau^{(n)} $, 
$v^{(n)}$, and $w^{(n)}$ such that
\begin{align*}
\left\vert
M^{(n; \star)}_{\lceil t^{(n)} \rceil^{(n)}}(v^{(n)}, w^{(n)}) - \Mart^{(n; \star)}_{t^{(n)}} ( \confmap_n (v^{(n)}) , \confmap_n (w^{(n)}) )
\right\vert
\geq \delta.
\end{align*}
By standard compactness arguments, we may extract a subsequence such that \linebreak $(\Gr^{(n)}_{ s }; e_1^{(s)}, \ldots, e_{2N}^{(s)} ; v^{(n)}, w^{(n)} ) $, with $s = \lceil t^{(n)} \rceil^{(n)}$, converge in the Carath\'{e}odory sense. Note that by the assumed setup, also $(\Gr^{(n)}; e_1^{(n)}, \ldots, e_{2N}^{(n)} ; v^{(n)}, w^{(n)} ) $ convergence in the Carath\'{e}odory sense. Consider now $n \to \infty$ along this subsequence. Using Theorem~\ref{prop: disc harm conv results}(iii) and~(iv) for $M^{(n; \star)}_{\lceil t^{(n)} \rceil^{(n)}}(v^{(n)}, w^{(n)})$ in~\eqref{eq: disc mg obs expanded}, and basic Carath\'{e}odory stability arguments for $\Mart^{(n; \star)}_{t^{(n)}} ( \confmap_n (v^{(n)}) , \confmap_n (w^{(n)}) )$, we observe that these two quantities then converge to the same limit, a contradiction.
\end{proof}

\subsubsection{Proof of martingaleness in Proposition~\ref{prop: cts martingale}}

$\Mart^{(\star)}_t( z , \omega)$ is clearly $\mathscr{F}_t$ adapted, and it is integrable since it is bounded. It remains to check the conditional expectation property. We claim that, for all $t \geq 0$,
\begin{align*}
\Mart^{(\star)}_t( z , \omega) = \limEX_\star [ \Mart^{(\star)}_\tau ( z , \omega) \; \vert \; \mathscr{F}_t],
\end{align*}
from which the conditional expectation property follows. Equivalently, we wish to show that for any $f_t$ continuous bounded function of $W$ measurable with respect to $\mathscr{F}_t$, we have
\begin{align}
\label{eq: desired cond exp}
\limEX_\star [ \Mart^{(\star)}_t ( z , \omega ) f_t (W) ] = \limEX_\star [ \Mart^{(\star)}_\tau ( z , \omega) f_t (W) ].
\end{align}
Let thus us prove~\eqref{eq: desired cond exp}. The proof is based on approximating the expectations on either side above by their discrete analogues. For notational simplicity, we will perform the analysis for the left-hand side --- the right-hand side can be treated analogously.

We would like to use the weak convergence $W^{(n)}_\cdot \to W_\cdot$. Note however that the process $\Mart^{(n;\star)}_t ( z , \omega)$ takes as input not only $W^{(n)}_\cdot$ but also the processes $X^{(n; i)}_t = g^{(n)}_t (X^{(n; i)}_0)$ and $( g^{(n)}_t )'  (X^{(n; i)}_0)$, with $i \ne j$. If we replaced them in the definition~\eqref{eq: def of disc-cont process} with $g^{(n)}_t (X^{(i)}_0)$ and $( g^{(n)}_t )'  (X^{(i)}_0)$, then $\Mart^{(n;\star)}_t ( z , \omega)$ and $\Mart^{(\star)}_t ( z , \omega)$ would both be simply the same continuous bounded function $h_{(t, z , \omega )}$ of the driving function, applied to $W^{(n)}_\cdot $ and $ W_\cdot$, respectively. (The boundedness uniformly over the driving function was proven in Lemma~\ref{lem: mgale boundedness} in Appendix~\ref{app: boundedness pf}, and continuity follows from the stability of the Loewner equation with respect to driving term.) Let us now compare $\Mart^{(n;\star)}_t ( z , \omega)$ and $h_{(t, z , \omega )} (W^{(n)}) $, i.e., replace $ g^{(n)}_t (X^{(n; i)}_0)$ and $( g^{(n)}_t )'  (X^{(n; i)}_0)$ by $g^{(n)}_t (X^{(i)}_0)$ and $( g^{(n)}_t )'  (X^{(i)}_0)$. First, changing $X^{(i)}_0$ to $X^{(n; i)}_0$ will perturb $g^{(n)}_t(\cdot)$ and $(g^{(n)}_t)'(\cdot)$ applied to these starting points by a small amount, uniformly over $t$ and $W^{(n)}$.\footnote{
For $(g^{(n)}_t)(\cdot)$, such a stability follows directly from Gr\"{o}nwall's lemma, similarly to Equation~\eqref{eq: Gronwall} in Appendix~\ref{app: boundedness pf}; using this stability and Gr\"{o}nwall's lemma again, one then obtains a similar stability for $(g^{(n)}_t)'(\cdot)$.
}
Second, by the compactness of the possible coordinates $g^{(n)}_t(X^{(i)}_0)$, proven in Appendix~\ref{app: boundedness pf}, also $1/\PartF_\star$ only acquires a small perturbation, again uniformly over $t$ and $W^{(n)}$ (recall that a continuous function is uniformly continuous on a compact set). In conclusion, we have
\begin{align}
\label{M as a cts fcn}
\Mart^{(\star)}_t ( z , \omega ) = h_{(t, z , \omega )} (W)
\qquad
\text{and}
\\
\label{Mn as a cts fcn}
\Mart^{(n; \star)}_t ( z , \omega ) = h_{(t, z , \omega )} (W^{(n)}) + o(1),
\end{align}
the latter asymptotic formula as $n \to \infty$, $o(1)$ small uniformly over $t$ and $W^{(n)}$. Altogether, we get
\begin{align*}
\limEX_\star & [ \Mart^{(\star)}_t ( z , \omega ) f_t (W) ]
\\
 \text{(use~\eqref{M as a cts fcn})} &=  \limEX_\star   [ h_{(t, z , \omega )} (W ) f_t (W) ] 
 \\
 \text{(weak conv.)} &=  \EX^{(n)}_\star  [ h_{(t, z , \omega )} (W^{(n)} ) f_t (W^{(n)}) ] 
 + o(1) \\
 \text{(use~\eqref{Mn as a cts fcn})} &=  \EX^{(n)}_\star [ \Mart^{(n; \star)}_t ( z , \omega ) f_t (W^{(n)}) + o(1)]
 + o(1)  \\
  \text{($o(1)$ uniform)} &=  \EX^{(n)}_\star  [ \Mart^{(n; \star)}_t ( z , \omega ) f_t (W^{(n)}) ] 
  + o(1).
\end{align*}
Note that here and in continuation it is important that the error terms inside the expectation operator are uniform, and can thus be taken outside of the expectation.

Next, let $v^{(n)} \in \Vert_n$ be the vertex for which $\confmap_n (v^{(n)}) $ is as close to $z$ as possible, and define $w^{(n)} \in \Vert_n$ closest to $\omega$ in the analogous sense. It is easy to deduce that $\Mart^{(n; \star)}_t ( z, \omega)$ is uniformly close to $\Mart^{(n; \star)}_t ( \confmap_n (v^{(n)}), \confmap_n (w^{(n)}))$ as $n \to \infty$. One then obtains
\begin{align}
\nonumber
\limEX_\star [ \Mart^{(\star)}_t ( z , \omega ) f_t (W) ] & =  
\EX^{(n)}_\star   [ \Mart^{(n; \star)}_t ( \confmap_n (v^{(n)}), \confmap_n (w^{(n)})) f_t (W^{(n)}) ] 
+ o(1) \\
\label{eq: cont to disc process 1}
\text{(Prop.~\ref{prop: unif conv of mgale obs})} &= \EX^{(n)}_\star  [ M^{(n; \star)}_{\lceil t \rceil^{(n)}  \wedge \lceil \tau^{(n)} \rceil^{(n)} }(v^{(n)}, w^{(n)} ) f_t (W^{(n)}) ] 
+ o(1),
\end{align}
where in both steps a uniform error term $o(1)$ was taken out of the expectation and absorbed into the previous one.

Repeating the argument of the previous two paragraphs for the right-hand side of~\eqref{eq: desired cond exp}, we get
\begin{align}
\label{eq: cont to disc process 2}
\limEX_\star [ \Mart^{(\star)}_\tau ( z , \omega) f_t (W) ] & =  
\EX^{(n)}_\star [ M^{(n; \star)}_{\lceil \tau^{(n)} \rceil^{(n)}}( v^{(n)}, w^{(n)} ) f_t (W^{(n)}) ]
+ o(1).
\end{align}
Now, notice that $f_t (W^{(n)})$ is measurable in the stopped sigma algebra $\mathcal{F}_{\lceil t \rceil^{(n)}  \wedge \lceil \tau^{(n)} \rceil^{(n)}}$ of the discrete time filtration $\mathcal{F}_s$ of the WST branch on $\Gr^{(n)}$. Also, for each fixed $n$, the stopping times $\lceil t \rceil^{(n)}  \wedge \lceil \tau^{(n)} \rceil^{(n)}$ and $ \lceil \tau^{(n)} \rceil^{(n)}$ are bounded, and thus we have
\begin{align*}
M^{(n; \star)}_{\lceil t \rceil^{(n)}  \wedge \lceil \tau^{(n)} \rceil^{(n)} }(z^{(n)}, \omega^{(n)} ) = 
\EX^{(n)}_\star \bigg[ M^{(n; \star)}_{\lceil \tau^{(n)} \rceil^{(n)}}( z^{(n)}, \omega^{(n)} ) \; \bigg\vert \; \mathcal{F}_{\lceil t \rceil^{(n)}  \wedge \lceil \tau^{(n)} \rceil^{(n)}} \bigg].
\end{align*}
With these two observations,~\eqref{eq: cont to disc process 2} yields
\begin{align}
\label{eq: cont to disc process 3}
\limEX_\star [ \Mart^{(\star)}_\tau ( z , \omega) f_t (W) ] & =  
\EX^{(n)}_\star [ M^{(n; \star)}_{\lceil t \rceil^{(n)}  \wedge \lceil \tau^{(n)} \rceil^{(n)} }(z^{(n)}, \omega^{(n)} ) f_t (W^{(n)}) ]
+ o(1).
\end{align}
Finally, combining~\eqref{eq: cont to disc process 1} and~\eqref{eq: cont to disc process 3} and taking the limit $n \to \infty$ proves~\eqref{eq: desired cond exp}. This finishes the martingaleness part, and the entire proof of Proposition~\ref{prop: cts martingale}.
{ \ }$\hfill \qed$

\subsection{Identification via martingales}
\label{subsubsec: ito exercise}

The second step in our proof of Theorem~\ref{thm: main result} is to use the martingales from the first step to identify $W$ via explicit It\^{o} calculus. In order to apply It\^{o}'s theorem, we first need to show the semimartingaleness of the driving function $W$.

\begin{lem}
\label{lem: semimg dr fcn}
The weak limit process $W$ is a semimartingale.
\end{lem}

\begin{proof}
The proof is based on applying the Implicit function theorem to the martingales of Proposition~\ref{prop: cts martingale}. Denote $g_t(z) = z_t$ and define a complex-valued process $f$ in terms of $ W_t$ and the time-differentiable processes $z_t$,  $X^{(i)}_t$, and $g_t' (X^{(i)}_0)$, where $i \neq j$, given by
\begin{align}
\label{eq: process f}
f(z_t; W_t; (X^{(i)}_t)_{i \neq j}; (g_t' (X^{(i)}_0))_{i \neq j})
 :=
&  \frac{ 1 }{ (z_t- W_t ) \PartF_\star (X^{(1)}_t,  \ldots,  X^{(2N)}_t) } \prod_{ \substack{ i = 1 \\ i \neq j} }^{2N} \frac{1}{ g_t' (X^{(i)}_0)} ,
\quad \text{stopped at } \tau.
\end{align}
Note that for $z \in \bH \setminus U_\epsilon$ and given $\omega \in \bH$, $\Mart^{(\star)}_t( z , \omega)$ is by~\eqref{eq: cts mgale} and~\eqref{eq: defn of cont poisson kernel} a constant multiple of $\Im (f)$.
Observe also that $f$ and $\partial_{W_t} f$ are complex analytic in $z_t$ (we will below treat $z_t \in \bH$ as the complex argument of $f$).

We first claim that for any function $ W_\cdot$ (and the related maps $g_\cdot$) and any $t < \tau$,  we have $\partial_{W_t} \Im (f) = \Im (\partial_{W_t}  f) \neq 0$ for almost every $z \in H_t$, or equivalently, for almost every $z_t \in \bH$. Indeed, by basic properties of analytic functions, either $ \Im (\partial_{W_t}  f) = 0$ for all $z_t \in \bH$, or $\Im (\partial_{W_t}  f) \neq 0$ for almost every $z_t \in \bH$. By explicit differentiation, we see that the latter occurs.

Take now a deterministic countable dense set of complex numbers $z$ in a fixed ball in $\bH \setminus U_\epsilon$. By the previous paragraph, for any $t < \tau$, we must have $\partial_{W_t} \Im (f) \neq 0$ for some of these $z$'s, and by continuity in time, $\partial_{W} \Im (f) \neq 0$ then also holds on some time interval around $t$.
Now, by the Implicit function theorem, whenever $\partial_{W_t} \Im (f) \neq 0$, the collection of local inverses of $\Im (f)$ provides a smooth function $\psi$ such that the relation
\begin{align*}
W_t = \psi(z_t, (X^{(i)}_t)_{i \neq j}; (g_t ' (X^{(i)}_0))_{i \neq j}, \Mart^{(\star)}_t( z , \omega))
\end{align*}
holds (here we also used $ (\Im (f))_t = cst. \times \Mart^{(\star)}_t( z , \omega)$).
In conclusion, we have deterministic collections of complex numbers $z$ and local inverses $\psi$ of $\Im (f)$, and it holds that for every $t < \tau$ there exist $z$ and $\psi$ in these collections such that the relation above is valid over some open time interval containing $t$. We now conclude by observing that $\psi(z_t, (X^{(i)}_t)_{i \neq j}; (g_t ' (X^{(i)}_0))_{i \neq j}, \Mart^{(\star)}_t( z , \omega))$ (when it is defined) is a semimartingale. Indeed, $\psi$ is smooth, $\Mart^{(\star)}_t ( z , \omega)$ is a continuous bounded martingale by Proposition~\ref{prop: cts martingale}, while $g_t (z)$, $g_t' (X^{( i)}_0 )$, and $X^{( i)}_t$ are differentiable in time by the Loewner equations~\eqref{eq: Loewner ODE} and~\eqref{eq: LE for derivative}. The claim thus follows by basic martingale theory.
\end{proof}

\begin{proof}[Proofs of Theorems~\ref{thm: main result} and~\ref{cor: CFT degeneracy PDEs}]
Let $f$ be as defined in~\eqref{eq: process f}.
As $W$ is a semimartingale, we can apply It\^{o}'s theorem to deduce that for $t < \tau$ (we omit writing the arguments of $f$)
\begin{align}
\label{eq: Ito formula}
\ud f = \frac{1}{2} \partial^2_{W_t} f \ud \langle W, W \rangle_t + 
\partial_{W_t} f \ud W_t + \partial_{z_t} f \ud z_t 
+\sum_{i \neq j} \partial_{X^{(i)}_t} f \ud X^{(i)}_t 
+ \sum_{i \neq j} \partial_{g_t'(X^{(i)}_0) } f \ud (g_t'(X^{(i)}_0))_t.
\end{align}
Let us now compute the various terms above (we also omit writing the arguments of $\PartF_\star$ and its derivatives):
\begin{small}
\begin{align}
\label{eq: stoch D 1}
\partial_{W_t} f \ud W_t &= \left( \frac{ 1 }{ z_t- W_t  }  -  \frac{ \partial_j  \PartF_\star  }{  \PartF_\star } \right) f \ud W_t \\
\label{eq: stoch D 2}
\frac{1}{2} \partial_{W_t}^{2} f  \ud \langle W, W \rangle_t &= \frac{1}{2}  \left( \frac{ 2 }{ (z_t- W_t )^2 }  + 2  \left( \frac{ \partial_j  \PartF_\star  }{  \PartF_\star } \right)^2  - \frac{ 2 \partial_j  \PartF_\star  }{ (z_t- W_t ) \PartF_\star } - \frac{ \partial_{jj}  \PartF_\star  }{ \PartF_\star } \right) f  \ud \langle W, W \rangle_t  \\
\label{eq: stoch D 3}
\partial_{z_t} f \ud z_t &= - \frac{ 1 }{ z_t- W_t  } f \frac{ 2 }{ z_t- W_t  } \ud t \\
\label{eq: stoch D 4}
\partial_{X^{(i)}_t} f \ud X^{(i)}_t &= - \frac{ \partial_{i}  \PartF_\star  }{ \PartF_\star  } f  \frac{ 2 }{ X^{(i)}_t - W_t  } \ud t \\\
\label{eq: stoch D 5}
\partial_{g_t'(X^{(i)}_0 ) } f \ud (g_t'(X^{(i)}_0 ))_t &= \frac{ 2 }{ (X^{(i)}_t - W_t  )^2 } f \ud t,
\end{align}
\end{small}%
where the three last equations also used the Loewner differential equations~\eqref{eq: Loewner ODE} and~\eqref{eq: LE for derivative}.

By~\eqref{eq: Ito formula}--\eqref{eq: stoch D 5} (and noticing that $f$ is a real scaling of $1  /(z_t- W_t ) $), each term of $\ud f$ is of the form 
\begin{align*}
\frac{ 1 }{ (z_t- W_t )^k } \times [\text{real stochastic differential}], \qquad
\text{where } k \in \{ 1, 2, 3 \}.
\end{align*}
Also, notice that $\Im (f)$ is a martingale by Proposition~\ref{prop: cts martingale} (i.e., the drift part of $\Im (\ud f)$ vanishes) \textit{simultaneously for all} $z \in \bH \setminus U_\epsilon$. It is easy to show that this can only occur if the drift stochastic differential coefficients of $1 / (z_t- W_t )^k $ cancel out \textit{for each $k$ individually}.

We now examine the different powers $1 / (z_t- W_t )^k $ individually. Terms of the form $f/(z_t- W_t )^2$ in~\eqref{eq: Ito formula} come from~\eqref{eq: stoch D 2} and~\eqref{eq: stoch D 3}, and impose
\begin{align}
\nonumber
\frac{1}{2}  \frac{ 2 f }{ (z_t- W_t )^2 } \ud \langle W, W \rangle_t - \frac{ 2 f }{ (z_t- W_t )^2 } \ud t &= 0 \\
\label{eq: mgale part}
\Leftrightarrow  \ud \langle W, W \rangle_t  &= 2 \ud t.
\end{align} 
Terms  $f/(z_t- W_t )$ appear in~\eqref{eq: stoch D 1} and~\eqref{eq: stoch D 2} and they yield
\begin{align}
\nonumber
\frac{ f }{ z_t- W_t  } \ud [ \text{drift part of }W ]_t - \frac{1}{2} \frac{ 2 f \partial_j  \PartF_\star  }{ (z_t- W_t ) \PartF_\star } \ud \langle W, W \rangle_t &= 0 \\
\label{eq: drift part}
\Leftrightarrow 
\ud [ \text{drift part of }W ]_t  &= 2 \frac{  \partial_j  \PartF_\star  }{  \PartF_\star } \ud t.
\end{align}
Equations~\eqref{eq: mgale part} and~\eqref{eq: drift part} are sufficient to identify the stochastic differential of the semimartingale $W$ as
\begin{align*}
\ud W_t &= \sqrt{ 2 } \ud B_t + 2 \frac{\partial_j \PartF_\star (X^{(1)}_t  \ldots  X^{(2N)}_t )}{\PartF_\star (X^{(1)}_t  \ldots  X^{(2N)}_t )} \ud t,
\end{align*}
i.e., $W_t$ is the $\SLE(2)$ driving function with the partition function $\PartF_\star$. This finishes the proof of Theorem~\ref{thm: main result}.

Let us yet prove Theorem~\ref{cor: CFT degeneracy PDEs}. This is based on collecting the terms of the form $f/(z_t- W_t )^0$ in~\eqref{eq: Ito formula} from~\eqref{eq: stoch D 1},~\eqref{eq: stoch D 2},~\eqref{eq: stoch D 4}, and~~\eqref{eq: stoch D 5}:
\begin{align*}
-  \frac{ \partial_j  \PartF_\star  }{  \PartF_\star } f \ud [ \text{drift part of }W ]_t 
+
\frac{1}{2} \left(  2  \left( \frac{ \partial_j  \PartF_\star  }{  \PartF_\star } \right)^2  - \frac{ \partial_{jj}  \PartF_\star  }{ \PartF_\star } \right) f  \ud \langle W, W \rangle_t  & \\
-
\sum_{i \neq j} \frac{ \partial_{i}  \PartF_\star  }{ \PartF_\star  } f  \frac{ 2 }{ X^{(i)}_t - W_t  } \ud t
+
\sum_{i \neq j} \frac{ 2 }{ (X^{(i)}_t - W_t  )^2 } f \ud t
&= 0 \\
- \frac{ \partial_{jj}  \PartF_\star  }{ \PartF_\star } \ud t 
+
\sum_{i \neq j} \left( - \frac{ \partial_{i}  \PartF_\star  }{ \PartF_\star  } \frac{ 2 }{ X^{(i)}_t - W_t  } \ud t
+
\frac{ 2 }{ (X^{(i)}_t - W_t  )^2 }  \ud t \right) &= 0.
\end{align*}
Since this must hold for any $t$ and any initial configuration of the points $(X^{(1)}_0,  \ldots,  X^{(2N)}_0)$, as well as for any $j$, Theorem~\ref{cor: CFT degeneracy PDEs} follows.
\end{proof}

\subsection{Alternative proof strategies}

The proof of Theorem~\ref{thm: main result} relied on a discrete martingale observable that was a discrete Girsanov transform of the one-branch martingale~\eqref{eq: classical LERW mgale}. The limit identification step was then identical to that in the one-curve case, given by the special case $\star=1$. It seems possible that also the other martingales of Lemma~\ref{lem: Z-ratios are cond probas} and Proposition~\ref{prop: discrete mgales collected} could be used for proving Theorem~\ref{thm: main result}, as originally suggested in~\cite{mie2}. We now briefly describe three alternative proof strategies that seem tractable; note that the discussion below is speculative.
\begin{itemize}
\item[i)] Take as an input the identification of the one-curve scaling limit $\limEX_1$ as an $\SLE(2)$. Using the discrete transform converting expectations  of $\mathcal{F}_t$-measurable functions from $\EX^{(n)}_1$ to $\EX^{(n)}_N$ or $\EX^{(n)}_\alpha$, Lemma~\ref{lem: Z-ratios are cond probas}, find the continuous transform converting expectations  $\mathscr{F}_t$-measurable functions from $\limEX_1$ to $\limEX_N$ or $\limEX_\alpha$. Apply Girsanov's theorem to convert the driving function from $\limEX_1$ to $\limEX_N$ or $\limEX_\alpha$. This strategy is applied in~\cite{KS18} for FK-Ising and percolation.
\item[ii)] Take as an input the identification of the one-curve scaling limit and Theorem~\ref{cor: CFT degeneracy PDEs}, as proven independently in~\cite[Theorem~4.1]{KKP}. Use the fourth and fifth martingale $Z_\beta / Z_\star$ of Proposition~\ref{prop: discrete mgales collected}. Here two technical difficulties arise. First, proving Lemma~\ref{lem: semimg dr fcn} requires detailed analysis of the derivatives of the martingale functions, and these derivatives may be zero (simultaneously for all $\beta$) at least if $N=2$. Second, in the proof of Theorem~\ref{thm: main result}, we need at least two processes to identify the two differentials $\ud \langle W, W \rangle_t $ and $ \ud [ \text{drift part of }W ]_t $, so we need many enough martingales $\PartF_\beta / \PartF_\star$ (scaling limits of $Z_\beta / Z_\star$), and we need to establish suitable linear independence type results for these martingales (especially $N=2$ is a problem again).
Ways to overcome these difficulties are:
\begin{itemize}
\item[a)] to argue absolute continuity with respect to the one-curve case, which implies that $W$ is a semimartingale and $\ud \langle W, W \rangle_t = 2 \ud t$; or
\item[b)] to modify the fourth and fifth martingale results of Proposition~\ref{prop: discrete mgales collected}: $Z_\beta $  may be taken to be a WST connectivity partition function of $2M > 2N$ boundary points, and $Z_\beta / Z_\star$ is still a $\PR_\star$ martingale. The freedom of choice of these extra boundary points would probably solve the problems named above.
\end{itemize}
\item[iii)] A completely different approach based on global multiple SLEs is outlined in an update added to~\cite[Conjecture~4.3]{KKP} after the publication of~\cite{PW, BPW}.
\end{itemize}

\bigskip

\section{An analogous result for a boundary-visiting branch}
\label{se: bdry visiting branch}

We now sketch the proof of a local identification of the scaling limit of a single spanning tree branch conditioned on boundary visits. The result and its proof are closely analogous to our main theorem~\ref{thm: main result}, and we trust that the reader can fill the details omitted here for the sake of brevity. In particular, this proof provides another example of how discrete partition functions can be used for transforming martingale observables, in this case from the usual branch to the boundary-visiting branch.

\subsection{Statement}

This generalization only addresses the (isoradial)  square lattice $ \Z^2$, and the WST on its subgraph $\Gr = (\Vert, \Edg)$ thus becomes the \emph{uniform random spanning tree with wired boundary conditions (UST)}. An edge $\hat{e} \in \Edg$ is called \emph{boundary-neighbouring} if it is between two interior vertices, but both of these interior vertices are adjacent to the boundary vertices $\Vert^\bdry$. A boundary branch in the UST is said to \emph{visit boundary} at $\hat{e}$ if it traverses through $\hat{e}$. In this section, we will consider a single UST boundary-to-boundary branch between the boundary edges $e_1 = \ein$ and $e_2 = \eout$, with the additional condition of visiting boundary at the boundary-neighbouring edges $\hat{e}_1, \ldots, \hat{e}_{N'}$ in an order $\omega$, see Figure~\ref{fig: bdry visits bijection}(left) for an illustration. (We will always assume that the order of visits $\omega$ is topologically possible.)

The scaling limits are characterized in the following setup. Let \linebreak
$(\Gr_n; \ein^{(n)}, \eout^{(n)}; \hat{e}_{1}^{(n)}, \ldots, \hat{e}_{N'}^{(n)})$ be simply-connected subgraphs of $\delta_n \Z^2$, where $\delta_n \small{ \stackrel{n \to \infty}{\longrightarrow} } 0$, with two marked boundary edges and $N'$ marked boundary-neighbouring edges. Assume that, as planar domains with marked boundary points, $\Gr_n$ are uniformly bounded and converge in the Carath\'{e}odory sense to a domain $(\domain; \pin, \pout; \hat{p}_1 \ldots, \hat{p}_{N'})$ with $(2 +N')$ distinct marked prime ends. Assume also that the boundary of both $\Gr_n$ and $\domain$ is locally a straight horizontal or vertical line in some fixed neighbourhoods of the boundary-visit locations $\hat{e}_{1}^{(n)}, \ldots, \hat{e}_{N'}^{(n)}$ and $\hat{p}_1 \ldots, \hat{p}_{N'}$.

Let $\confmap_n: \domain_n \to \bH$ and $\confmap: \domain \to \bH$ conformal maps such that $\confmap_n^{-1} \to \confmap^{-1}$ uniformly over compact subsets of $\bH$. 
Denote $\confmap (\pin, \pout; \hat{p}_1 \ldots, \hat{p}_{N'}) =(X^{(\mathrm{in})}_0, X^{(\mathrm{out})}_0 ; \hat{X}^{(1)}_0,  \ldots  \hat{X}^{(N')}_0)$, and assume that $\confmap$ is chosen so that these prime ends of $\bH$ are all real (finite). Fix a localization neighbourhood $U$ of $X^{(\mathrm{in})}_0$ bounded away from the remaining marked boundary points $ X^{(\mathrm{out})}_0, \hat{X}^{(1)}_0,  \ldots  \hat{X}^{(N')}_0$.

Consider now WST boundary-to-boundary branch from $\ein^{(n)}$ to $\eout^{(n)}$ on $\Gr_n$, conditioned to visit boundary at $\hat{e}_{1}^{(n)}, \ldots, \hat{e}_{N'}^{(n)}$ in the (possible) order $\omega$. Map this branch conformally to $\bH$ by the map $\confmap_n$ above. Let $W^{(n)}_\cdot$ denote the driving functions in the Loewner evolutions describing the growth of the boundary-to-boundary branch starting from $\ein^{(n)}$ and stopped at the continuous modification $\tau^{(n)}$ of the exit time of $U$.

\begin{thm}
\label{thm: bdry visit convergence}
In the setup and notation above, $W^{(n)}_\cdot$ converge weakly to the SLE type driving function~\eqref{eq: defn of ptt fcn SLE}, stopped at $\tau$, with parameter $\kappa = 2$, and partition function $\zeta_\omega = \zeta_\omega (W_t, X^{(\mathrm{out})}_t, \hat{X}^{(1)}_t,  \ldots  \hat{X}^{(N')}_t) $ as given in~\cite[Theorem~1.1]{KKP}.
\end{thm}

The scaling limit above can interpreted as the initial segment of $\SLE(2)$ in $(\bH; X^{(\mathrm{in})}_0, X^{(\mathrm{out})}_0)$, conditioned to visit $\hat{X}^{(1)}_0, \ldots, \hat{X}^{(N')}_0$ in the order $\omega$, see Appendix~\ref{app: bdry vis SLE}.

\subsection{The combinatorial model}

Let us again start from the combinatorial solution. Consider the UST measure $\PR$ on $\Gr$, a simply-connected subgraph of $\Z^2$  (equipped with a choice of boundary vertices). Denote by
\begin{align*}
Z_\omega ( \ein, \eout; \hat{e}_{1}, \ldots, \hat{e}_{N'} ) 
\end{align*}
the probability that the boundary branch from $\ein^\circ$ reaches $\bdry \Vert$ via $\eout$ and visits $\hat{e}_{1}^{(n)}, \ldots, \hat{e}_{N'}^{(n)}$ in the order $\omega$. There is a bijection between spanning trees satisfying this condition, and spanning trees with $N=(N'+1)$ boundary-to-boundary branches that form a link pattern $\alpha = \alpha(\omega) \in \LP_N$ between the boundary edges $\tilde{e}_{1}, \ldots, \tilde{e}_{2N} $ obtained by re-labelling $ \ein$, $\eout$, and the $2N'$ boundary edges adjacent to $\hat{e}_{1}, \ldots, \hat{e}_{N'}$. Informally, the bijection is simply obtained by ``cutting the boundary-visiting branch at each visit'', see Figure~\ref{fig: bdry visits bijection}. For a formal description, see~\cite[Lemma~3.2]{KKP}.

\begin{figure}
\centering
\begin{overpic}[width=0.4\textwidth]{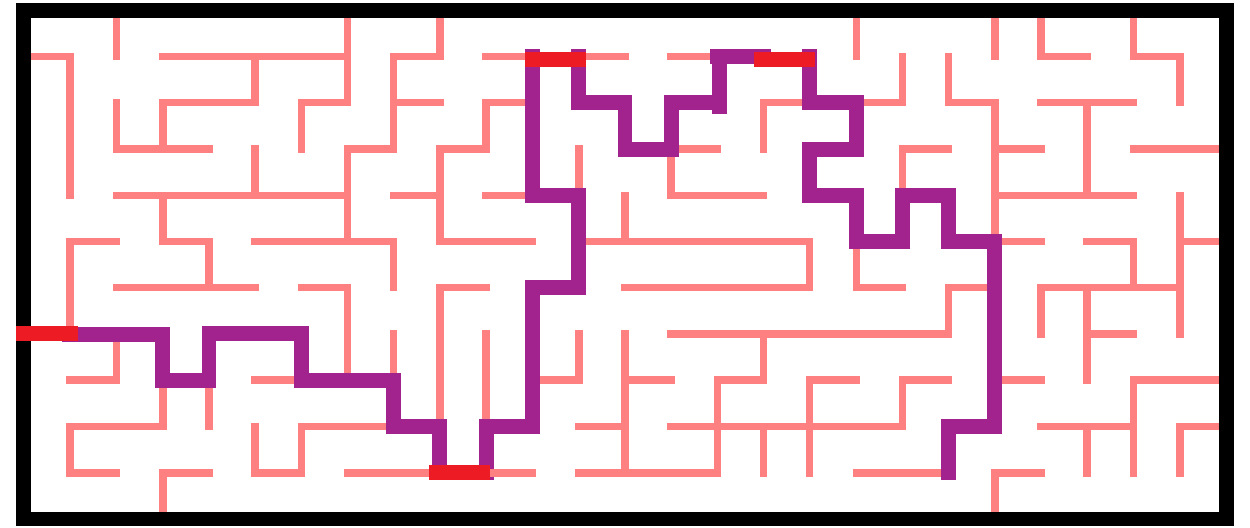}
 \put (-10,15) {\large $\eout$}
 \put (35,-6) {\large $\hat{e}_3$}
  \put (73,-5.5) {\large $\ein$}
 \put (61,44) {\large $\hat{e}_1$}
\put (42,44) {\large $\hat{e}_2$}
\end{overpic}
\qquad \qquad
\begin{overpic}[width=0.4\textwidth]{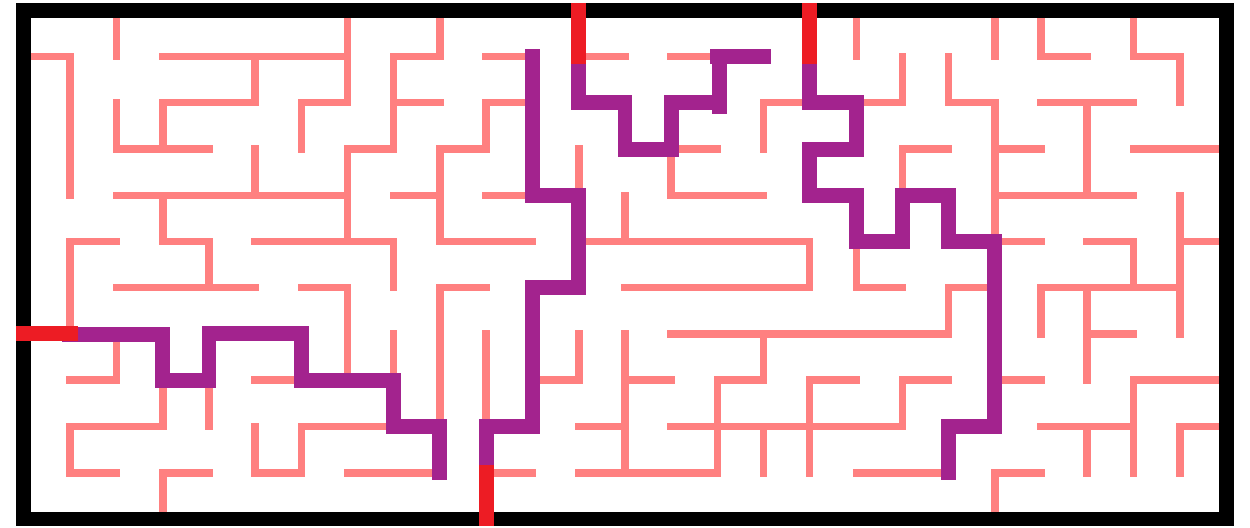}
 \put (-6,15) {\large $\tilde{e}_6$}
 \put (32,-6) {\large $\tilde{e}_7 \tilde{e}_8$}
  \put (73,-6) {\large $\tilde{e}_1$}
 \put (58,44) {\large $\tilde{e}_3 \tilde{e}_2$}
\put (39,44) {\large $\tilde{e}_5 \tilde{e}_4$}
\end{overpic}%
\caption{
Left: A UST sample with a boundary-to-boundary branch between the boundary edges $e_1 = \ein$ and $e_2 = \eout$, visiting boundary at the boundary-neighbouring edges $\hat{e}_1, \hat{e}_2, \hat{e}_{3}$ in that order. Left and right: An illustration of the bijection between spanning trees with a boundary-visiting boundary-to-boundary branch and those with multiple boundary-to-boundary branches.}
\label{fig: bdry visits bijection}
\end{figure}
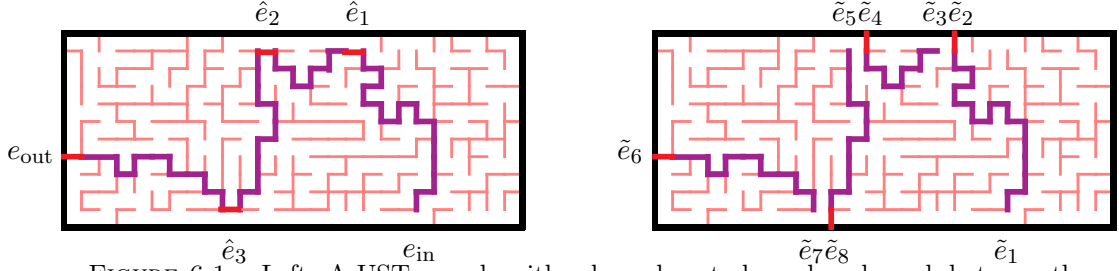

In particular, we have
\begin{align}
\label{eq: disc bdry vis ptt fcn}
Z_\omega ( \ein, \eout; \hat{e}_{1}, \ldots, \hat{e}_{N'} ) = Z_{ \alpha(\omega) } ( \tilde{e}_{1}, \ldots, \tilde{e}_{2N} ).
\end{align}
and also the initial segment of the boundary-visiting branch coincides with that of a suitable branch in the link pattern $\alpha (\omega) $ between $\tilde{e}_{1}, \ldots, \tilde{e}_{2N}$. 
Let us hence study the UST measure conditional on the multiple branches forming $\alpha(\omega)$ between $\tilde{e}_{1}, \ldots, \tilde{e}_{2N}$, denoted $\PR_\alpha$. A discrete martingale observable under $\PR_\alpha$ is given Proposition~\ref{prop: discrete mgales collected},
\begin{align}
\label{eq: bdry vis disc mgale}
M^{(\alpha)}_t(v, u ) &= \frac{\PoissonK^{\Gr_{t } } (v, \ein^{(t)} )  }{ \tilde{Z}^{ \Gr_{t } }_\alpha }  \PoissonK (u, \eout )  \qquad \text{stopped at $T$}.
\end{align}
The same martingale observable could be found directly under the boundary visiting branch, by modifying the proof of Proposition~\ref{prop: discrete mgales collected} so that martingales are transformed from a single branch to a boundary-visiting branch, and then using~\eqref{eq: disc bdry vis ptt fcn}.

\subsection{Observable convergence}

The expression for the discrete martingale observable~\eqref{eq: bdry vis disc mgale} and its scaling limit were studied in~\cite{KKP}: in the notation of Theorems~\ref{prop: disc harm conv results} and~\ref{thm: bdry visit convergence}
\begin{align*}
\delta^{3N'} & \frac{\PoissonK^{\Gr_{t } } (v, \ein^{(t)} )  }{ \tilde{Z}^{ \Gr_{t } }_\alpha }  \PoissonK (u, \eout ) 
\stackrel{ \delta \to 0 }{\longrightarrow}  
\left( \prod_{i=1}^{N'} g_t' ( \hat{X}^{(i)}_0 )^{-3} \vert \confmap'( \hat{p}_i ) \vert^{-3} \right) \frac{\PoissonKH (g_t ( z ), W_t )  }{  \zeta_\omega (W_t, X^{(\mathrm{out})}_t, \hat{X}^{(1)}_t,  \ldots  \hat{X}^{(N')}_t)  }   \frac{\PoissonKH ( w, X^{(\mathrm{out})}_0 ) }{ g_t ' ( X^{(\mathrm{out})}_0  ) },
\end{align*}
where $\confmap_\delta (v) \to z$ and  $\confmap_\delta (u) \to w$ as $\delta \to 0$, and $\zeta_\omega $ is the function given in~\cite[Theorem~1.1]{KKP}. (To be very precise, we need to adapt~\cite[Theorem~1.1]{KKP} a little bit to allow potentially rough boundaries at $\ein$ and $\eout$ above. This follows however by a simple application Theorem~\ref{prop: disc harm conv results}(iv).)

\subsection{Precompactness}

The precompactness of the multiple branches in the link pattern $\alpha (\omega) $ is proven identically to the precompactness part in the main theorem~\ref{thm: main result}, see Section~\ref{subsubsec: precompactness}. Note that in the main theorem, all the endpoints of WST branches tend to different limiting prime ends, while this is not the case here. Nevertheless, the precompactness conditions of~\cite{mie2}, checked in Section~\ref{subsubsec: precompactness}, guarantee precompactness even if some limiting prime ends coincide, see~\cite[Section~4.1]{mie2}.

\subsection{Continuous martingales in the scaling limit}

Extract now a subsequential weak limit driving function $W$. Repeating the arguments of Section~\ref{subsec: cts mgales} in verbatim, one observes that 
\begin{align}
\label{eq: bdry visits - cts mgale}
\frac{\PoissonKH (g_t ( z ), W_t )  }{  \zeta_\omega (W_t, X^{(\mathrm{out})}_t, \hat{X}^{(1)}_t,  \ldots  \hat{X}^{(N')}_t)  g_t ' ( X^{(\mathrm{out})}_0  ) }  \prod_{i=1}^{N'} g_t' ( \hat{X}^{(i)}_0 )^{-3} , \qquad \text{stopped at } \tau
\end{align}
is a bounded continuous martingale for all $z \in \bH \setminus U_\eps$.

\subsection{Identifying the scaling limit}

With a blue copy of Lemma~\ref{lem: semimg dr fcn} (but this time based on the martingale~\eqref{eq: bdry visits - cts mgale}), one proves that the driving function $W$ is a semimartingale.

We now finish the proof outline of Theorem~\ref{thm: bdry visit convergence} by identifying the law of $W$ with an explicit computation that closely resembles the proof Theorem~\ref{thm: main result} in Section~\ref{subsubsec: ito exercise}. We start by defining, analogously to~\eqref{eq: process f}, the process
\begin{align}
\label{eq: process f2}
f &(z_t; W_t; (\hat{X}^{(i)}_t)_{i \leq N'}; \hat{X}^{( \mathrm{out} )}_t; (g_t' ( \hat{X}^{(i)}_0 ))_{i \leq N'}; g_t' ( \hat{X}^{( \mathrm{out} )}_0 )) \\
\nonumber
& :=  \frac{ 1 }{ (z_t- W_t ) \zeta_\omega (W_t, X^{(\mathrm{out})}_t, \hat{X}^{(1)}_t,  \ldots  \hat{X}^{(N')}_t)  g_t ' ( X^{(\mathrm{out})}_0  )  } \prod_{  i = 1  }^{N'} \frac{1}{ g_t' ( \hat{X}^{(i)}_0 )^{3} } , \qquad \text{stopped at } \tau,
\end{align}
whose imaginary part $\Im f$ coincides with the martingale~\eqref{eq: bdry visits - cts mgale} up to a constant factor.

By It\^{o}'s theorem, we have for $t < \tau$ (omitting the arguments of $f$)
\begin{small}
\begin{align}
\nonumber
\ud f =
& \frac{1}{2} \partial^2_{W_t} f \ud \langle W, W \rangle_t + 
\partial_{W_t} f \ud W_t + \partial_{z_t} f \ud z_t 
\\
\nonumber
&
+\left( \sum_{i \leq N'} \partial_{ \hat{X}^{(i)}_t} f \ud \hat{X}^{(i)}_t + \partial_{ X^{( \mathrm{out} )}_t} f \ud \hat{X}^{(\mathrm{out} )}_t \right)
+ \partial_{g_t'( \hat{X}^{( \mathrm{out} )}_0) } f \ud (g_t'( \hat{X}^{( \mathrm{out} )}_0))_t
\\
\label{eq: Ito formula 2}
&
+ \sum_{i \leq N'} \partial_{g_t'( \hat{X}^{(i)}_0) } f \ud (g_t'( \hat{X}^{(i)}_0))_t,
\end{align}
\end{small}%
and the drift part of $\Im (\ud f )$ should vanish by the martingaleness of $\Im f$.

The five first terms, on the two first lines of~\eqref{eq: Ito formula 2}, yield, identically to the five terms in the It\^{o} differential~\eqref{eq: Ito formula}, 
\begin{small}
\begin{align}
\label{eq: 2 stoch D 1}
\partial_{W_t} f \ud W_t &= \left( \frac{ 1 }{ z_t- W_t  }  -  \frac{ \partial_{\xin}  \zeta_\omega  }{  \zeta_\omega } \right) f \ud W_t \\
\label{eq: 2 stoch D 2}
\frac{1}{2} \partial_{W_t}^{2} f  \ud \langle W, W \rangle_t &= \frac{1}{2}  \left( \frac{ 2 }{ (z_t- W_t )^2 }  + 2  \left( \frac{ \partial_{\xin}  \zeta_\omega  }{  \zeta_\omega  } \right)^2  - \frac{ 2 \partial_{\xin}  \zeta_\omega }{ (z_t- W_t ) \zeta_\omega } - \frac{ \partial_{\xin}^2  \zeta_\omega  }{ \zeta_\omega } \right) f  \ud \langle W, W \rangle_t  \\
\label{eq: 2 stoch D 3}
\partial_{z_t} f \ud z_t &= - \frac{ 1 }{ z_t- W_t  } f \frac{ 2 }{ z_t- W_t  } \ud t \\
\label{eq: 2 stoch D 4}
\partial_{X^{(  \mathrm{out}  )}_t} f \ud X^{(  \mathrm{out} )}_t &= - \frac{ \partial_{ \xout }  \zeta_\omega }{ \zeta_\omega  } f  \frac{ 2 }{ X^{(  \mathrm{out}  )}_t - W_t  } \ud t  \quad \text{and} \quad
\partial_{\hat{X}^{(i)}_t} f \ud \hat{X}^{(i)}_t = - \frac{ \partial_{ \hat{x}_i }  \zeta_\omega }{ \zeta_\omega  } f  \frac{ 2 }{ \hat{X}^{(i)}_t - W_t  } \ud t 
\\
\label{eq: 2 stoch D 5}
\partial_{g_t'(X^{( \mathrm{out} )}_0 ) } f \ud (g_t'(X^{( \mathrm{out} )}_0 ))_t &= \frac{ 2 }{ (X^{( \mathrm{out} )}_t - W_t  )^2 } f \ud t,
\end{align}
\end{small}%
while the last term of~\eqref{eq: Ito formula 2} becomes
\begin{small}
\begin{align}
\label{eq: 2 stoch D 6}
\partial_{g_t'( \hat{X}^{(i)}_0) } f \ud (g_t'( \hat{X}^{(i)}_0))_t &= \frac{ 6 }{ ( \hat{X}^{(i)}_t - W_t  )^2 } f \ud t.
\end{align}
\end{small}%

As in the proof Theorem~\ref{thm: main result} in Section~\ref{subsubsec: ito exercise}, one now argues that the real drift differentials multiplying the complex number $1/(z_t - W_t)^k$ must cancel out in~\eqref{eq: Ito formula 2} for each $k$ individually. Terms of the form $f/(z_t- W_t )^2$ come from~\eqref{eq: 2 stoch D 2} and~\eqref{eq: 2 stoch D 3}, and yield
\begin{align}
\label{eq: 2 mgale part}
  \ud \langle W, W \rangle_t  &= 2 \ud t,
\end{align} 
Terms  $f/(z_t- W_t )$ appear in~\eqref{eq: 2 stoch D 1} and~\eqref{eq: 2 stoch D 2} and yield
\begin{align}
\label{eq: 2 drift part}
\ud [ \text{drift part of }W ]_t  &= 2 \frac{  \partial_{\xin} \zeta_\omega   }{ \zeta_\omega  } \ud t.
\end{align}
Equations~\eqref{eq: 2 mgale part} and~\eqref{eq: 2 drift part} identify the semimartingale $W$ as the solution to
\begin{align*}
\ud W_t &= \sqrt{ 2 } \ud B_t + 2 \frac{ ( \partial_{\xin} \zeta_\omega )  (W_t, X^{(\mathrm{out})}_t, \hat{X}^{(1)}_t,  \ldots  \hat{X}^{(N')}_t)  }{ \zeta_\omega (W_t, X^{(\mathrm{out})}_t, \hat{X}^{(1)}_t,  \ldots  \hat{X}^{(N')}_t)  } \ud t,
\end{align*}
i.e., $W_t$ is the $\SLE(2)$ driving function with the partition function $\zeta_\omega$. This proves Theorem~\ref{thm: bdry visit convergence}.

Terms of the form $f/(z_t- W_t )^0$ in~\eqref{eq: Ito formula 2} provide a nice double-check: the come from~\eqref{eq: 2 stoch D 1},~\eqref{eq: 2 stoch D 2},~\eqref{eq: 2 stoch D 4},~\eqref{eq: 2 stoch D 5}, and~\eqref{eq: 2 stoch D 6}, and yield
\begin{small}
\begin{align*}
-  \frac{ \partial_{\xin}  \zeta_\omega  }{  \zeta_\omega } f \ud [ \text{drift part of }W ]_t 
+
\frac{1}{2} \left(  2  \left( \frac{ \partial_{\xin}  \zeta_\omega  }{  \zeta_\omega  } \right)^2  - \frac{ \partial_{\xin}^2  \zeta_\omega  }{ \zeta_\omega } \right) f  \ud \langle W, W \rangle_t 
 -
 \frac{ \partial_{ \xout }  \zeta_\omega }{ \zeta_\omega  } f  \frac{ 2 }{ X^{(  \mathrm{out}  )}_t - W_t  } \ud t
 & \\
 - \sum_{i \leq N'} \frac{ \partial_{ \hat{x}_i }  \zeta_\omega }{ \zeta_\omega  } f  \frac{ 2 }{ \hat{X}^{(i)}_t - W_t  } \ud t 
+
\frac{ 2 }{ (X^{( \mathrm{out} )}_t - W_t  )^2 } f \ud t
+
 \sum_{i \leq N'}  \frac{ 6 }{ ( \hat{X}^{(i)}_t - W_t  )^2 } f \ud t
& = 0 
\end{align*}
\end{small}%
Similarly to the proof of Theorem~\ref{cor: CFT degeneracy PDEs} in Section~\ref{subsubsec: ito exercise}, one deduces that $\zeta_\omega $ must satisfy the PDE
\begin{align}
\nonumber
- \partial_{\xin}^2  \zeta_\omega (\xin, \xout; \hat{x}_1, \ldots, \hat{x}_{N'})
-
\frac{ 2 }{ \xout - \xin  } \partial_{ \xout }  \zeta_\omega (\xin, \xout; \hat{x}_1, \ldots, \hat{x}_{N'}) & \\
\nonumber
 - \sum_{i \leq N'} \frac{ 2 }{ \hat{x}_i - \xin  }  \partial_{ \hat{x}_i }  \zeta_\omega  (\xin, \xout; \hat{x}_1, \ldots, \hat{x}_{N'})  
 + \frac{ 2 }{ (\xout -  \xin   )^2 }  \zeta_\omega  (\xin, \xout; \hat{x}_1, \ldots, \hat{x}_{N'}) & \\
\label{eq: bdry vis PDE}
 + \sum_{i \leq N'} \frac{ 6 }{ ( \hat{x}_i - \xin )^2 }   \zeta_\omega  (\xin, \xout; \hat{x}_1, \ldots, \hat{x}_{N'})  & = 0.
\end{align}
The exact same PDE for $\zeta_\omega$ was proven in~\cite[Theorem~1.1]{KKP} with a completely different method.

We conclude by remarking that the core of~\cite[Theorem~1.1]{KKP}, certain third order PDEs of Conformal field theory for $\zeta_\omega$, do \textit{not} arise from this probabilistic study, as anticipated in~\cite{KKP}.


\bigskip

\appendix

\section{On the boundary behaviour of discrete harmonic functions and the precompactness of WST branches}
\label{app: disc harm}

The main result of this appendix is Theorem~\ref{thm: ratios of discrete harmonic functions converge}, relating the boundary behaviour of discrete harmonic functions to that of the continuous ones. The main observable convergence result of this paper, Theorem~\ref{prop: disc harm conv results}, is a simple application of Theorem~\ref{thm: ratios of discrete harmonic functions converge}. A key ingredient the proof of Theorem~\ref{thm: ratios of discrete harmonic functions converge} is a Beurling type estimate for random walk \emph{excursions} from~\cite{KS}, recalled in Proposition~\ref{lem: crossing an annulus is a uniformly nontrivial event}. Interestingly, this proposition also constitutes the proof of precompactness of the multiple WST boundary-to-boundary branches.

\subsection{Discrete harmonic functions}
\label{subsec: basic notions of disc harm fcns}

Consider for a moment the setup of Section~\ref{sec: connectivity probas}, i.e., $\Gr=(\Vert, \Edg)$ is a finite connected planar graph with a planar embedding, a choice of boundary vertices, and edge weights $\EdgeWeight$. The \emph{discrete Laplacian} $\Delta$ maps a function $f: \Vert \to \R$ on the vertices to another function $\Delta f: \Vert \to \R$ given by
\begin{align*}
\Delta f (v) = 
\sum_{e = \edgeof{v}{u} \in \Edg} \EdgeWeight(e)(f(u)-f(v)).
\end{align*}
The function $f$ is \emph{discrete harmonic} is $\Delta f (v) = 0$ for all interior vertices $v \in \Vert^\circ$.

Recall the definitions of the discrete Green's function and Poisson and excursion kernels from Section~\ref{subsubsec: discrete kernels}. Note that the Green's function $\GreenK (\cdot, w)$, interpreted as a function $\Vert \to \R$ with fixed $w \in \Vert$, satisfies $\GreenK (v, w) = 0$ for all $v \in \bdry \Vert $, and for $v \in \Vert^\circ$,
\begin{align*}
\Delta \GreenK (v, w) = - \delta_w (v) = 
\begin{cases}
-1, \qquad v = w \\
0, \qquad v \neq w.
\end{cases}
\end{align*}
By linearity, Poisson problems of the discrete Laplacian can thus be solved in terms of $\GreenK (\cdot, w)$.

For an interior vertex $v \in \Vert^\circ$ and a set of boundary edges $A \subset \bdry \Edg$, we define the \emph{discrete harmonic measure of $A$ as seen from $v$}, denoted by $\HarmMeas_\Gr (v; A)$, as the probability that random walk on $\Vert$ with edge weights $\EdgeWeight$, launched from $v \in \Vert$, first reaches $\bdry \Vert$ via an edge of $A$.
All harmonic measures can be expressed in terms of Poisson kernels; if $A=\{ e \} $ consists of a single boundary edge, then
\begin{align*}
\HarmMeas_\Gr (v; \{ e \}) = \EdgeWeight(e) \PoissonK (v, e),
\end{align*}
and otherwise $\HarmMeas_\Gr (v; A) = \sum_{e \in A} \HarmMeas_\Gr (v; \{ e \})$.

Assume now that all boundary edges of $\Gr$ link to different boundary vertices (or modify $\Gr$ accordingly). The harmonic measure $\HarmMeas_\Gr (v; A)$ (or the Poisson kernel) can be regarded a discrete harmonic function on $\Gr$ by extending it to $w \in \bdry \Vert$ by setting $\HarmMeas_\Gr (w; A) = \mathbbm{1} \{ w $ adjacent to $A\}$. (This follows as Poisson kernel is a Green's function and thus harmonic except at one vertex.) This justifies the term harmonic measure, and is crucial for the scaling limit analysis of Poisson kernels.

\subsection{A Beurling type estimate for random walk excursions}

The main result of this subsection, Proposition~\ref{lem: crossing an annulus is a uniformly nontrivial event}, is central in the proof of Theorem~\ref{thm: ratios of discrete harmonic functions converge} and implies directly the precompactness of the WST boundary-to-boundary branches. It was first given in~\cite[Section 4.5]{KS} in the latter purpose.
Due to its double importance, we recall the argument of~\cite{KS} here.
Note that in the results in this subsection are \emph{uniform} over all simply-connected subgraphs of isoradial lattices (possibly different lattices, satisfying~\ref{eq: unif bound on half-angles}), and do not depend on the mesh size.

\subsubsection{Random walks and isoradial balls and quadrilaterals}

Let $\InfiniteGr$ be an isoradial lattice with the isoradial edge weights $\EdgeWeight$. We will denote the probability measure of the $\EdgeWeight$-weighted random walk on $\InfiniteGr$, launched from $v \in \Vert (\InfiniteGr)$, by $\PR_{v}$. Let $\Gr = (\Vert, \Edg)$ be a simply-connected subgraph of $\InfiniteGr$. A \emph{random walk excursion} from an interior vertex $v \in \Vert^\circ$ to some boundary edges $A \subset \bdry \Edg$ on $\Gr$ is a random walk launched from $v$, stopped upon hitting the boundary vertices $\bdry \Vert$, and conditioned to first reach $\bdry \Vert$ via an edge of $A$. We will denote the underlying random walk measure with this conditioning by $\PR^\Gr_{\pathfromto{v}{A} }$.

We define the \emph{discrete ball $B_\InfiniteGr(u, \rho)$ around a vertex} $u \in \Vert(\InfiniteGr)$, for any $\rho > 0$, as the following simply-connected isoradial graph. Take the largest simple loop $\ell$ on the dual lattice $\InfiniteGr^*$ such that the primal vertices inside $\ell$ include $u$ and are contained in the (continuous) ball $B(u,\rho)$. The primal vertices inside $\ell$ are the interior vertices of $B_\InfiniteGr(u, \rho)$, all their other neighbours in $\InfiniteGr$ are the boundary vertices of $B_\InfiniteGr(u, \rho)$, and the edges of $B_\InfiniteGr(u, \rho)$ are those of $\InfiniteGr$ between the vertices of $B_\InfiniteGr(u, \rho)$. We will later refer to a simply-connected isoradial subgraphs constructed via a dual loop $\ell$, similarly to the above, as the \emph{subgraph of $\InfiniteGr$ determined by $\ell$}.

Let $\Gr = (\Vert, \Edg)$ be a simply-connected isoradial subgraph of $\InfiniteGr$. By an \emph{isoradial quadrilateral $Q$ on the boundary of $\Gr$}, we mean the following topological quadrilateral. We take two disjoint simple paths $\ell$ and $\ell'$ on the dual isoradial lattice $\InfiniteGr^*$, both crossing the boundary of $\Gr$ (as viewed as a continuous domain) by their first and last edges and otherwise staying inside the domain $\Gr$. We also require that there exists a path on the graph $\Gr$ crossing both $\ell$ and $\ell'$ and staying on the interior vertices $\Vert^\circ$ except for possibly at its endpoints. Then, the planar domain $Q$ is\footnote{
The definition of $Q$ is here slightly more restrictive than in~\cite{KS}. This is in order to directly comply with assumptions in a result of~\cite{Chelkak-Robust_DCA_toolbox}, used in the proof of Lemma~\ref{lem: Kemppainen lemma}.
}
 the unique connected component of the planar domain $\Gr \setminus (\ell \cup \ell')$ adjacent to both $\ell$ and $\ell'$. Two segments of $\bdry Q$ lie on the arcs $\ell$ and $\ell'$ and two on the boundary of the domain $\Gr$, giving the structure of a topological quadrilateral. We denote by $m(Q)$ the \emph{conformal modulus} of $Q$, i.e., the unique $L>0$ such that $Q$ can be conformally mapped to the rectangle $(0,1) \times (0,L)$, where the top and bottom arcs correspond to $\ell$ and $\ell'$.

Let $v \in \Vert^\circ$ be an interior vertex and $A \subset \bdry \Edg$ a set of boundary edges of $\Gr$. We say that $Q$ is \emph{compatible} with $A$ (resp. $v \cup A$), if all edges of $A$ lie at least half outside of the planar domain $Q$ (resp. and also $v \not \in Q$), and these edges or half-edges of $A \setminus Q$ (resp. and also $v$) all lie in one component of the planar domain $\Gr \setminus Q$. For definiteness, we will assume the $\ell$ separates $\ell'$ from $v$ and the edges or half-edges of $A \setminus Q$.

Consider now the simply-connected subgraph of $\InfiniteGr$ determined by the largest simple loop on the dual $\InfiniteGr^*$, so that the dual loop stays in $Q \cup \ell \cup \ell'$ and intersects both $\ell$ and $\ell'$.
Slightly abusively, we will also denote by $Q$ this subgraph of $\InfiniteGr$ and $\Gr$.
The boundary vertices $\Vert^{\bdry} (Q)$ are naturally divided into four disjoint subsets $S_0, S_1, S_2, S_3$ indexed counterclockwise: $S_0$ and $S_2$ are adjacent to boundary edges crossing the dual paths $\ell$ and $\ell'$, respectively, while $S_1$ and $S_3$ are adjacent to the remaining two arcs of the simple dual loop.
Note that a walk on the interior vertices $\Vert^\circ$ that crosses the domain $Q$ from $\ell$ to $\ell'$ must make the crossing on the (\textit{a priori} smaller) graph $Q$ from $S_0$ to $S_2$.

\subsubsection{A Beurling estimate for random walk excursions}

\begin{lem}
 \label{lem: Kemppainen lemma}
\cite[ Proposition 4.17]{KS} 
There exists $\eps_0 > 0$ such that for any $c > 0$ there exists $M>0$ such that the following holds. Let $\Gr = (\Vert , \Edg)$ be any simply-connected isoradial subgraph with any marked boundary edges $A \subset \bdry \Edg$.
Let $Q$ be an isoradial quadrilateral on the boundary of $\Gr$, compatible with $A$.
Now, if $m(Q) \geq M$, then there exists $u \in \Vert^\circ(Q)$ and $\rho > 0$ such that
\begin{itemize}
\item[i)] $B:= B_\InfiniteGr(u, \rho)$ a subgraph of the graph $ Q$;
\item[ii)] $\min_{x \in  B} \HarmMeas_\Gr (x; A) \ge c \max_{x \in S_2 } \HarmMeas_\Gr (x; A)$; and
\item[iii)] $\PR^{Q}_{ \pathfromto{x}{e} } [$excursion intersects $ B ] \ge \eps_0$ for all $x \in \Vert^\circ(Q)$ adjacent to $S_0$ and $e \in \bdry \Edg (Q)$ adjacent to $S_2$.
\end{itemize}
\end{lem}

\begin{proof}
The proof in~\cite[Proposition~4.17]{KS} is given for $\InfiniteGr = \Z^2$. 
Their argument, as well as its inputs (the weak Beurling estimate~\cite[Proposition~2.11]{CS-discrete_complex_analysis_on_isoradial},
the Harnack lemma~\cite[Proposition~2.7]{CS-discrete_complex_analysis_on_isoradial}, and~\cite[Proposition~3.3]{Chelkak-Robust_DCA_toolbox}), apply directly to general isoradial graphs $\InfiniteGr$.
\end{proof}

The proof of the following lemma essentially coincides with the proof of~\cite[Theorem 4.18]{KS}, where the statement is however given in terms of loop-erased random walks.

\begin{proposition} 
\emph{(cf. \cite[Theorem 4.18]{KS})}
\label{lem: crossing an annulus is a uniformly nontrivial event}
Continue in the setup and notation of the above lemma, and let $c$ be chosen larger than $1$.  Let $a \in \Vert^\circ$ be any vertex such that $Q$ is compatible with $a \cup A$. Denote by $\tau$, $\tau'$, and $T$ the first times the random walk (trajectory) from $a$ hits $\ell$, $\ell'$, and $\bdry \Vert$, respectively. Then, we have
\begin{align}
\label{eq: crossing condn}
\PR_{ \pathfromto{a}{A} }^{\Gr} [\tau' \leq T \vert \tau \leq T ] \le \frac{1}{\eps_0 (c-1)} \times \PR_{ \pathfromto{a}{A} }^{\Gr} [\tau' > T \vert \tau \leq T ].
\end{align}
In particular, for any $\tilde{\eps} > 0$, by choosing $M$ large enough, having $m(Q) \geq M$ guarantees that
\begin{align}
\label{eq: crossing proba}
\PR_{ \pathfromto{a}{A} }^{\Gr} [\tau' \leq T ] < \tilde{\eps}.
\end{align}
\end{proposition}

\begin{remark}
By applying~\eqref{eq: crossing condn} to a ``rainbow of discrete boundary quadrilaterals'' one can improve~\eqref{eq: crossing proba} to
\begin{align*}
\PR_{ \pathfromto{a}{A} }^{\Gr} [\tau' \leq T ] \leq K e^{-\alpha m (Q)}
\end{align*}
for some positive absolute constants $K, \alpha$. This can be seen as an analogue of the weak Beurling estimate for the random walk~(e.g.,~\cite[Proposition~2.11]{CS-discrete_complex_analysis_on_isoradial}) for excursions of random walks. This extension is nevertheless not necessary here, and we leave the details to the reader. (The non-trivial part is to divide a big discrete quadrilateral into a rainbow of smaller ones so that the latter are \textit{discrete} and~\eqref{eq: crossing condn} remains valid.)
\end{remark}

\begin{proof}[Proof of Proposition~\ref{lem: crossing an annulus is a uniformly nontrivial event}]
Consider a random walk excursion from  $a$ to $A$. If this excursion crosses $\ell'$, decompose it into three subwalks: roughly, in $\Gr$ and not crossing $\ell'$ from $a$ to $S_0$, then in $Q$ from $S_0$ to $S_2$, and then in $\Gr$ from $S_2$ to $A$. Formally, the last part starts after the first crossing of $\ell'$ and the middle part after the last crossing of $\ell$ before it. Decompose the partition function of random walk excursions from $a$ to $A$ accordingly. Lemma \ref{lem: Kemppainen lemma}~(iii) now implies that the middle part of this decomposition visits $B$ with probability at least $\eps_0$, so
\begin{align}
\nonumber
\PR_{ \pathfromto{a}{A} } [\tau' \leq T \vert \tau \leq T ] & \le \frac{1}{\eps_0 } \PR_{ \pathfromto{a}{A} } [\tau_B \leq \tau' \leq T \vert \tau \leq T ] \\
\label{eq: excursion Beurling step}
&= \frac{ \PR_{ \pathfromto{a}{A} } [\tau_B \leq \tau' \leq T  ]  }{\eps_0 \PR_{ \pathfromto{a}{A} } [  \tau \leq  T ]  },
\end{align}
where $\tau_B $ denotes the hitting time of $B$.
(Inside this proof, all excursion probabilities are in $\Gr$, so we drop the superscripts.) Next, we claim that
\begin{align}
\label{eq: exc Beurling step 2}
\PR_{ \pathfromto{a}{A} } [\tau_B \leq \tau' \leq T  ] \le \frac{1}{c-1} \PR_{ \pathfromto{a}{A} } [ \tau_B  \leq  T < \tau'   ] ,
\end{align}
where $c$ is the constant in Lemma \ref{lem: Kemppainen lemma}~(ii). To prove this, divide the excursion from  $a$ via first $B$ and then $S_2$ to $A$ (resp. via $B$ but avoiding $S_2$) into two subwalks: 
first from $a$ until first hitting $B$, and then from $B$ via $S_2$ (resp. avoiding $S_2$) to $A$. Studying the latter parts, observe that for a random walk $\eta$ launched from $y \in B$,
\begin{align*}
\PR_{ \pathfromto{y}{A} } [ \tau' \leq T  ]  &= \frac{\PR_{ y } [ \tau'  \leq T \ \& \ \edgeof{\eta(T-1) }{\eta(T)} \in A  ]  }{ \PR_{ y } [  \edgeof{\eta(T-1) }{\eta(T)} \in A  ]  } \\
& \le \frac{ \max_{x \in S_2} \HarmMeas_\domain (x; A)  }{ \min_{x \in  B} \HarmMeas_\domain (x; A)  }  \\
& \le 1/c,
\end{align*}
where we used Lemma \ref{lem: Kemppainen lemma}~(ii).
This also implies $\PR_{ \pathfromto{y}{A} } [ \tau' > T  ] \ge 1 - 1/c$. Taking $c > 1$ in Lemma \ref{lem: Kemppainen lemma}, we now obtain \eqref{eq: exc Beurling step 2}. Finally, combining \eqref{eq: excursion Beurling step} and \eqref{eq: exc Beurling step 2}, we obtain
\begin{align*}
\PR_{ \pathfromto{a}{A} } [\tau' \leq T \vert \tau \leq T ]
& \le \frac{1}{\eps_0 (c-1)} \frac{  \PR_{ \pathfromto{a}{A} } [ \tau_B  \leq  T < \tau'   ]   }{  \PR_{ \pathfromto{a}{A} } [  \tau \leq T ]   } \\
&\le \frac{1}{\eps_0 (c-1)} \ \PR_{ \pathfromto{a}{A} } [  \tau' > T  \vert  \tau \leq T ] . 
\end{align*}
This proves \eqref{eq: crossing condn}.

To prove~\eqref{eq: crossing proba}, notice that the excursion from $a$ to $A$ has to cross $\ell$ to reach $\ell'$. Thus, 
\begin{align*}
\PR_{ \pathfromto{a}{A} }^{\Gr} [\tau' \leq T ] 
&\leq
\PR_{ \pathfromto{a}{A} }^{\Gr} [\tau' \leq T \vert \tau \leq T ].
\end{align*}
If $c$ is chosen very large, then by~\eqref{eq: crossing condn} the conditional probability on the right-hand side above is much smaller than the conditional probability of its complement event. Thus, the right-hand side above is small.
\end{proof}

\subsection{On the precompactness of WST boundary-to-boundary branches}
\label{subsubsec: precompactness}

Consider now a random walk excursion from a boundary-neighbouring vertex $e_1^\circ$, where $e_1 = \edgeof{e_1^\circ}{e_1^\bdry} \in \bdry \Edg$, to a given boundary edge $e_2 \in \bdry \Edg$. The loop-erasure of this excursion is in distribution equal to a WST boundary-to-boundary branch from $e_1$ to $e_2$~(see~\cite{Wilson-generating_random_spanning_trees} and~\cite[Corollary~3.5]{KKP}).
Proposition~\ref{lem: crossing an annulus is a uniformly nontrivial event} now states that for this random walk excursion and any compatible isoradial quadrilateral $Q$ with $m(Q)$ large enough, we have
 \begin{align}
 \label{eq: LERW crossing proba bound}
\PR_{ \pathfromto{ e_1^\circ }{ e_2 } }^{\Gr} [\text{excursion crosses }Q] < \tilde{\eps}.
\end{align}
The loop-erasure of the excursion (i.e., the WST branch) is of course even less likely to cross $Q$.

Crossing probability bounds as~\eqref{eq: LERW crossing proba bound} above, uniform over all planar domains and quadrilaterals, were in~\cite{KS} shown to imply the precompactness of a chordal random curve model (see also~\cite{mie}), in this case the WST boundary-to-boundary branch. Assuming the crossing probability bound for a single boundary-to-boundary branch, a similar bound was derived for multiple branches in~\cite{mie2}. These arguments in~\cite{KS, mie, mie2} are not specific for the WST but hold for a wide range of random curve models; in this sense Proposition~\ref{lem: crossing an annulus is a uniformly nontrivial event} is the core of the precompactness WST boundary-to-boundary branches.
For the remaining details on the precompactness, we refer to~\cite[Proof of precompactness in Theorem~6.8]{mie2}.

\subsection{Ratios of discrete harmonic functions near a Dirichlet boundary}

Consider now a sequence of isoradial graphs $\InfiniteGr_n$ with mesh sizes $\delta_n \to 0$. Let $\Gr_n = (\Vert_n, \Edg_n)$ be a sequence simply-connected isoradial subgraphs such that, as planar domains, $\Gr_n \to \domain$ in the sense of Carath\'{e}odory, with the corresponding conformal maps $\confmapH_n : \domain_n \to \bH $ and $\confmapH : \domain \to \bH $, where $\confmapH_n  \to \confmapH$ uniformly on the compact subsets of $\domain$. Let the interior vertices $a_n \in \Vert_n^\circ$ approximate the prime end $a $ of $\domain$ in the sense that $a_n' =: \confmapH_n (a_n) \to a' =:  \confmapH(a)$, and assume that $a_n$ are connected to (the vertex of $\Vert_n$ closest to) the reference point of the Carath\'{e}odory convergence by a path on the interior vertices $\Vert^\circ_n$.

\begin{thm}
\label{thm: ratios of discrete harmonic functions converge}
In the setup above, let $f_n, g_n : \Vert_n \to \C$ be non-negative discrete harmonic functions converging to the harmonic functions $f, g : \domain \to \C$, respectively, uniformly over compact subsets of $\domain$. Assume furthermore that $f_n, g_n$ attain zero boundary values in the image of a neighbourhood of $a'$ under $\confmapH_n^{-1}$, for all $n$. Then, we have
\begin{align*}
\frac{f_n (a_n)}{g_n (a_n)} \to \frac{\partial_y  (f \circ \confmapH^{-1} )  (a') }{ \partial_y (g \circ \confmapH^{-1} )  (a')  }
\end{align*}
if $\partial_y (g \circ \confmapH^{-1} )  (a')  \ne 0$, where $\partial_y$ denotes the vertical derivative (which is well defined at $a'$ by Schwarz reflection). 
\end{thm}

\begin{remark}
\label{rem: ratios of disc harm fcns}
The theorem above holds even if $f_n$ and $g_n$ only converge uniformly over compact subsets $\confmapH^{-1}(V)$, where $V$ is some neighbourhood of $a'$ in $\bH$; this is seen by simply applying it to $f_n$ and $g_n$ restricted to suitable subgraphs of $\Gr_n$.
\end{remark}

One way to prove Theorem~\ref{thm: ratios of discrete harmonic functions converge} would be based on~\cite[Corollary~3.8]{CW-mLERW}, see~\cite[Proof of Proposition~3.14]{CW-mLERW} for an analogue. We present below a proof based on Proposition~\ref{lem: crossing an annulus is a uniformly nontrivial event}.

\begin{proof}[Proof of Theorem~\ref{thm: ratios of discrete harmonic functions converge}]
Roughly, the idea of the proof is to consider $f_n (a_n)$ and $g_n (a_n)$ as expectations with respect to the random walk launched from $a_n$. Then, we split the walks into a beginning in the vincinity of the boundary point $a$ and a tail after exiting that vincinity for a first time.

Formally, let $A'= \bdry B(a', R) \cap \bH$ be a semicircle of radius $R$ around $a'$ in $\bH$, where $R$ is a (small) fixed number. Let $r < R$ be another (very small) fixed number and denote 
\[
A'_r= \{ w \in A' \ : \Im w > r \}.
\]
 We assume that $R$ is small enough, so that the curve $\bdry B(a', 2R) $ separates $\confmapH_n ( a_n )$ from the nonzero boundary values of $f_n$ and $g_n$ (as mapped to $\bH$). The index $n$ will be assumed large enough so  the edges of $\confmapH_n (\Gr_n)$ inside or intersecting the half-circle $A'$ have radii $< r$ (this is guaranteed for large enough $n$ by standard harmonic measure arguments), and $d(\confmapH_n (a_n), a') < r$. Finally, denote by $A_n, A_{r; n} \subset \domain_n $ the images in $\domain_n$ of the  sets $A'$ and $A'_r$, respectively, under $\confmapH_n^{-1}$.

Let $\tau_{\bdry \Vert_n}$ denote the hitting time of $\bdry \Vert_n$ by a random walk, and define $\tau_n = \min \{ \tau_{A_{n}}, \tau_{\bdry \Vert_n} \}$ and $\tau_{r, n} = \min \{ \tau_{A_{r, n}}, \tau_{\bdry \Vert_n} \}$, where $\tau_{A_{ n}}$ and $\tau_{A_{r; n}}$ are the hitting times of $A_n$ and $ A_{r; n} $, respectively, by the walk trajectory.
Notice that if $\tau_n =\tau_{A_n} < \infty$, the walk is at that time at a distance $\le \mesh_n$ from $A_n$. We denote the set of all such possible vertices as $A_n^{\InfiniteGr}$.  We define analogously  the vertex set $A_{r; n}^{\InfiniteGr}$ related to $A_{r; n}$. 

Let us now perform an analysis on $f_n$ --- identical conclusions hold for $g_n$. First, expressing $f_n (a_n)$ as an expectation
with respect to a random walk $\eta$ launched from $a_n$, and then using the fact that $A_n$ separates $a_n$ from the nonzero boundary values of $f_n$, we have
\begin{align}
\nonumber
f_n (a_n) &= \sum_{\substack{b \in \bdry \Vert_n \\ f(b) \ne 0 } } f_n (b) \PR^{a_n} [\eta(\tau_{\bdry \Vert_n } ) = b ] \\
\label{Exc steps}
&=    \sum_{\substack{b \in \bdry \Vert_n \\ f_n (b) \ne 0 } } f_n (b)  \sum_{v \in A_n^{\InfiniteGr}}  \PR^{a_n} [\eta(\tau_{\bdry \Vert_n } ) = b \  \& \ \eta(\tau_{n}) = v ].
\end{align}
Next, we use Proposition~\ref{lem: crossing an annulus is a uniformly nontrivial event}: study the random walk $\eta$, launched from $a_n$ and conditioned on hitting $\bdry \Vert_n$ at a fixed boundary vertex $b$ where $f_n$ is nonzero. If this walk crosses the curve segments $A_n \setminus A_{r; n}$, its image in $\bH$ crosses a half-annulus centered at $a' \pm R$, with inner and outer radii given by $2r$ and $R-r$ (recall that we have discretization errors $\le r$), respectively.
Proposition~\ref{lem: crossing an annulus is a uniformly nontrivial event} guarantees that the probability of such a crossing is $o(1)$ as $r/R \to 0$, the error term small \emph{uniformly} over the mesh sizes $\mesh_n$ and $b \in \bdry \Vert_n$ such that $f_n (b) \ne 0$. We denote the error term by $o_{r/R}(1)$, to keep explicit the variable in the Landau notation. With the positivity of $f_n$, this implies
 \begin{align}
 \label{Exc steps 2}
 \tag{\ref{Exc steps} b}
f_n (a_n) &= (1 + o_{r/R}(1))
\sum_{\substack{b \in \bdry \Vert_n \\ f_n (b) \ne 0 } } f_n (b)   \sum_{v \in A_{r; n}^{\InfiniteGr}} \PR^{a_n} [\eta(\tau_{\bdry \Vert_n } ) = b \  \& \ \eta(\tau_{n}) = v  ].
\end{align}

Next, notice that, using the separating property of $A_n$ and the strong Markov property of the random walk, we have
\begin{align*}
\PR^{a_n} [\eta(  \tau_{\bdry \Vert_n }  ) = b \ \& \ \eta(\tau_{n }) = v ] &=
\PR^{a_n} [ \eta(\tau_{ n }) = v ] \PR^{v} [\eta(  \tau_{\bdry \Vert_n }  ) = b  ].
\end{align*}
So \eqref{Exc steps 2} becomes
 \begin{align}
\nonumber
f_n (a_n) &= (1 + o_{r/R}(1))
 \sum_{v \in A_{r; n}^{\InfiniteGr}} \PR^{a_n} [ \eta(\tau_{ n }) = v ] \sum_{\substack{b \in \bdry \Vert_n \\ f(b) \ne 0 } } f_n (b)  \PR^{v} [\eta(  \tau_{\bdry \Vert_n }  ) = b  ] \\
 \nonumber
 &=  (1 + o_{r/R}(1))
 \sum_{v \in A_{r; n}^{\InfiniteGr} } \PR^{a_n} [ \eta(\tau_{ n }) = v ] f_n(v ) \\
   \label{Exc steps 3}
 \tag{\ref{Exc steps} c}
 &= (1 + o_{r/R}(1))   \PR^{a_n}  [\eta(\tau_{n }) \in A_{r; n}^{\InfiniteGr} ]  \EX^{a_n}  [f_n ( \eta(\tau_{n }) ) \ \vert \ \eta(\tau_{n }) \in A_{r; n}^{\InfiniteGr}  ].
\end{align}

We now study the function $f_n (v)$ in~\eqref{Exc steps 3}, where $v \in A_{r; n}^{\InfiniteGr}$. Recall first that $f_n$ converges to $f$ uniformly on the compact subsets of $\domain$, so $f_n (v) = f(v) + o_n^{(r, R)}(1)$ as $n \to \infty$;  by uniformity, the error term $o_n^{(r, R)}(1)$ is small uniformly over $v \in A_{r; n}^{\InfiniteGr}$ once $r $ and $R$ are fixed --- we add the superscript $(r, R)$ to the Landau notation to indicate the dependency. As a next step, let $\tilde{v}$ be the point of $A_{r; n}$ closest to the vertex $v$, so $d(\tilde{v}, v) \leq \mesh_n$. Since $f$ is continuous and thus uniformly continuous on a compact set, we can now substitute $f (v ) $ with $ f(\tilde{v})$, making an error of $o_n^{(r, R)}(1)$ again small uniformly over $v$, i.e., $f (v) = f(\tilde{v}) + o_n^{(r, R)}(1)$. Finally, since $f$ has zero boundary values at $a$, the harmonic function $f \circ \confmapH^{-1}: \bH \to \R$ extends by Schwarz reflection to a smooth function in a neighbourhood of $a' = \confmapH (a) \in \R$. Thus, we can approximate $f(\tilde{v}) = f \circ \confmapH^{-1} ( \confmapH (\tilde{v})  )$ in the previous expression by the Taylor expansion of $f \circ \confmapH^{-1}$ at $a'$, giving $f(\tilde{v}) = \Im (\confmapH (\tilde{v}) )  \partial_y \left( f \circ \confmapH^{-1} \right) (a')  + O_R (R^2)$, where the error term $O_R(R^2)$ only depends on the function $f \circ \confmapH^{-1}$. Altogether, this chain of approximations yields
\begin{align}
\label{eq: expand f_n}
f_n (v) = \Im (\confmapH (\tilde{v}) )  \partial_y \left( f \circ \confmapH^{-1} \right) (a') + O_R (R^2) + o_n^{(r, R)}(1),
\end{align}
where the error terms are small uniformly over $v \in A_{r; n}^{\InfiniteGr}$.

Now, substitute~\eqref{eq: expand f_n} back into~\eqref{Exc steps 3} and use the uniformity of the Landau terms of~\eqref{eq: expand f_n} in $v \in A_{r; n}^{\InfiniteGr}$:
 \begin{align}
f_n (a_n) 
\nonumber
 =  (1 + o_{r/R}(1))   \PR^{a_n}  [\eta(\tau_{n }) \in A_{r; n}^{\InfiniteGr} ]  \bigg( & \partial_y \left( f \circ \confmapH^{-1} \right) (a')  \EX^{a_n}  [ \Im (\confmapH (\tilde{v}))  \ \vert \ v = \eta(\tau_{n }) \in A_{r; n}^{\InfiniteGr}  ] 
 \\ & + O_R (R^2) + o_n^{(r, R)}(1) \bigg).
    \label{Exc steps 4}
 \tag{\ref{Exc steps} d}
\end{align}

Finally, use the expansion \eqref{Exc steps 4} for both $f_n$ and $g_n$. We have
\begin{small}
 \begin{align}
\frac{f_n (a_n) }{  g_n (a_n)}
 &= \frac{    (1 + o_{r/R}(1))   \left( \partial_y \left( f \circ \confmapH^{-1} \right) (a')  \EX^{a_n}  [ \Im (\confmapH (\tilde{v}))  \ \vert \ v = \eta(\tau_{n }) \in A_{r; n}^{\InfiniteGr}  ] + O_R (R^2) + o_n^{(r, R)}(1) \right)   }{    (1 + o_{r/R}(1))   \left( \partial_y \left( g \circ \confmapH^{-1} \right) (a')  \EX^{a_n}  [ \Im (\confmapH (\tilde{v}))  \ \vert \ v = \eta(\tau_{n }) \in A_{r; n}^{\InfiniteGr}  ] + O_R (R^2) + o_n^{(r, R)}(1) \right)  }.
\end{align}
\end{small}%
Notice that $\Im (\confmapH (\tilde{v_n})) \geq r$ by definition. Choosing for instance $R=\eps^2$ and $r=\eps^3$, and then taking first $\eps$ small enough and then $n$ large enough, we observe that $f_n (a_n) / g_n (a_n)$ can be made arbitrarily close to 
\begin{align*}
\frac{\partial_y  (f \circ \confmapH^{-1} )  (a') }{ \partial_y (g \circ \confmapH^{-1} )  (a')  }.
\end{align*}
This proves the claim.
\end{proof}

\subsection{Proof of Theorem~\ref{prop: disc harm conv results}}
\label{subsec: proof of disc harm conv res}

Parts~(i) and~(ii) are proven in~\cite[Corollary~3.11]{CS-discrete_complex_analysis_on_isoradial} and~\cite[Theorem~3.13]{CS-discrete_complex_analysis_on_isoradial}, respectively.
Part~(iii) follows by applying Theorem~\ref{thm: ratios of discrete harmonic functions converge} (see also Remark~\ref{rem: ratios of disc harm fcns}) to the fraction
\begin{align*}
 \frac{ \left( \PoissonK^{(n)} (u, e^{(n)}_2) \bigg/ \PoissonK^{(n)} (w^{(n)}, e^{(n)}_2) \right) }{ \GreenK^{(n)} (v^{(n)}, u )  }.
\end{align*}
This is a ratio of two positive discrete harmonic functions of $u$, both converging uniformly by parts~(i) and~(ii). 
 Part~(iv) follows by applying Theorem~\ref{thm: ratios of discrete harmonic functions converge} to
 \begin{align*}
 \frac{ \GreenK^{ \Gr^{(n)} } (v^{(n)}, u )}{ \GreenK^{ \tilde{\Gr}^{(n)} } (w^{(n)}, u)}.
 \end{align*}
For part~(iv), notice also the covariance formula for the boundary derivatives on Green's functions under conformal mapping-out, derived in the end of Section~\ref{subsubsec: cts kernels}.
{ \ }$\hfill \qed$

\bigskip

\section{Proof of boundedness in Proposition~\ref{prop: cts martingale}}
\label{app: boundedness pf}

We will actually prove a slightly stronger statement: the martingale~$\Mart^{(\star)}_t( z , \omega) $ in~\eqref{eq: cts mgale} is bounded, and this does not rely on any knowledge about the weak limit $W$ but only on how~$\Mart^{(\star)}_t( z , \omega) $ is defined in terms of a given continuous driving function. This small difference will turn important later when we want to use the weak convergence $W^{(n)} \to W$.

\begin{lem}
\label{lem: mgale boundedness}
Let $\Mart^{(\star)}_t( z , \omega) $ be as in~\eqref{eq: cts mgale} but constructed from any continuous driving function $W$ (not necessarily on the support of the weak limit). For any $z$ and $\omega$, there exists $C > 0$ such that for all continuous functions $W$ and all $t \le \tau$, we have
\begin{align*}
\big\vert \Mart^{(\star)}_t( z , \omega) \big\vert < C.
\end{align*}
\end{lem}

\begin{proof}
We deduce an upper bound for the right-hand side of~\eqref{eq: cts mgale}, uniform over $t \le \tau$ and over the driving function $W$. The factors $\PoissonKH (  \omega, X^{(i)}_0)$ are deterministic so they may be omitted in this analysis.

Let us first lower-bound $\Im ( g_t(z) )$ and hence upper-bound $\PoissonKH ( g_t(z) , W_t )$ (see~\eqref{eq: defn of cont poisson kernel}). Fix any $w \in \bH \setminus U_\epsilon$. It follows from Loewner's equation that $\Im (g_t (w)) \geq \Im (g_{U_\epsilon} (w))$, where $g_{U_\epsilon}$ is the mapping-out function of $U_\epsilon$. 
Clearly, we can lower-bound the harmonic measure of \textit{any} curve from $z$ to $\bdry H_t$, as seen from $w$ in $H_t$, by the probability of a Brownian motion from $w$ looping around $z$ before exiting $\bH \setminus U_\epsilon$. By conformal invariance, also the harmonic measure of the straight vertical line segment from $g_t(z)$ to $\R$, as seen from $g_t(w)$ in $\bH$, has the same lower bound. On the other hand, since we have $\Im (g_t (w)) \geq \Im (g_{U_\epsilon} (w)) $, the distance from $g_t (w)$ to this vertical line segment is at least $\Im (g_{U_\epsilon} (w)) - \Im (g_t(z)) $, and Beurling's estimate now upper-bounds the same harmonic measure by (essentially) a power of $\Im (g_t(z)) / \Im (g_{U_\epsilon} (w))$. For the lower and upper bounds to be consistent, $\Im ( g_t(z) )$ cannot be too small compared to $\Im (g_{U_\epsilon} (w)) $. 

Let us next lower-bound $\PartF_\star (X^{(1)}_t,  \ldots,  X^{(2N)}_t)$. By translation invariance (see~\eqref{eq: cont ptt fcns - conditional}--\eqref{eq: cont ptt fcns - unconditional}), we may do it assuming $X^{(1)}_t = 0$. The strategy is to first find both upper and lower bounds for the differences $( X^{(2)}_t - X^{(1)}_t ), \ldots, ( X^{(2N)}_t - X^{(2N-1)}_t ) $.  Then, assuming $X^{(1)}_t = 0$ and these upper and lower bounds, we obtain a compact set of coordinates $(X^{(1)}_t,  \ldots,  X^{(2N)}_t)$, and the continuous function $\PartF_\star (X^{(1)}_t,  \ldots,  X^{(2N)}_t)$ attains a minimum under these assumptions. This minimum is positive by~\cite{KKP}. It remains to find the upper and lower bounds for $( X^{(2)}_t - X^{(1)}_t ), \ldots, ( X^{(2N)}_t - X^{(2N-1)}_t ) $. By an argument very similar to the previous paragraph, one can lower-bound the harmonic measures of all the intervals $(-\infty, X^{(1)}_t), (X^{(1)}_t, X^{(2)}_t), \ldots, (X^{(2N)}_t, + \infty)$, as seen from $g_t(w)$ in $\bH$. Since $\Im (g_t (w)) \geq \Im (g_{U_\epsilon} (w))$, Beurling's estimate now shows that $( X^{(2)}_t - X^{(1)}_t ), \ldots, ( X^{(2N)}_t - X^{(2N-1)}_t )$ are lower-bounded. For the upper bound, suppose for a contradiction that we had $( X^{(2N)}_t -  X^{(1)}_t )/2 \geq C \Im (w)$ for a large enough $C$. From Loewner's equation we have $\Im (g_t (w)) \leq \Im (w)$. Thus, the circular annulus centered at $\Re (g_t (w)) \in \R$ with inner radius $\Im (w)$ and outer radius $C \Im (w)$ disconnects $g_t(w)$ from either $(X^{(2N)}_t, + \infty)$ (if $\Re (g_t (w)) \leq ( X^{(2N)}_t +  X^{(1)}_t )/2$) or $(-\infty, X^{(1)}_t)$ (if $\Re (g_t (w)) \geq ( X^{(2N)}_t +  X^{(1)}_t )/2$). Beurling's estimate then upper-bounds the harmonic measure of one of these intervals by (essentially) a power of $1/C$. If $C$ is large enough, this contradicts the previously derived constant lower bound, so we must have $( X^{(2N)}_t -  X^{(1)}_t ) \leq 2 C \Im (w)$.

We are left with lower-bounding $g_t' (X^{(i)}_0)$. Notice that by~\eqref{eq: LE for derivative}
\begin{align*}
\partial_t g_t ' (X^{(i)}_0) & = - \frac{ 2 g_t ' (X^{(i)}_0) }{ (X^{(i)}_t - X^{(j)}_t )^2}.
\end{align*}
In the previous paragraph, we deduced $\vert X^{(i)}_t - X^{(j)}_t \vert \geq C$. It thus follows that
\begin{align}
\label{eq: Gronwall}
\partial_t g_t ' (X^{(i)}_0) & \geq - \frac{ 2  }{ C^2} g_t ' (X^{(i)}_0),
\end{align}
and by Gr\"{o}nwall's lemma and the initial condition $ g_0 ' (X^{(i)}_0) = 1$, we have $g_t ' (X^{(i)}_0) \geq \exp(-2t/C^2)$. Since the stopping time $\tau$ is less than the half-plane capacity of $U_\epsilon$, we obtain a lower bound for $g_t ' (X^{(i)}_0)$. This concludes the proof of boundedness.
\end{proof}

\bigskip

\section{Boundary-visiting SLEs}
\label{app: bdry vis SLE}

The scaling limit in Theorem~\ref{thm: bdry visit convergence} can interpreted as the initial segment of $\SLE(2)$ in $(\bH; X^{(\mathrm{in})}_0, X^{(\mathrm{out})}_0)$ conditioned to visit $\hat{X}^{(1)}_0, \ldots, \hat{X}^{(N')}_0$ in the order $\omega$. We review here the non-rigorous argument leading to this interpretation, following~\cite{JJK-SLE_boundary_visits, KKP}.

Recall that the $\SLE(\kappa)$ in $(\bH; \xin, \xout)$ is a random curve $\gamma$ defined as a conformal image of an $\SLE(\kappa)$ curve in $(\bH; 0, \infty)$. The initial segment of the curve $\gamma$ from $\xin$ to $\xout$, up to the stopping time $\tau$ as in Theorem~\ref{thm: bdry visit convergence}, can almost surely be described by a Loewner chain, namely the partition function $\SLE(\kappa)$ with~\cite{Dubedat-commutation}
\begin{align}
\label{eq: SLE SDE}
\begin{cases}
\ud W_t = \sqrt{\kappa} \ud B_t + \kappa \frac{\partial_{W} \PartF (W_t, X^{(\mathrm{out})}_t )}{ \PartF (W_t, X^{(\mathrm{out})}_t ) } \ud t \\
\ud X^{(\mathrm{out})}_t = \frac{2}{ X^{(\mathrm{out})}_t - W_t } \ud t ,
\end{cases}
\end{align}
where $W_0 = \xin$, and $X^{(\mathrm{out})}_0 = \xout$, and
\begin{align*}
\PartF (a, b) = \vert b - a \vert^{1-6/ \kappa}.
\end{align*}
Denote the probability measure of $\SLE(\kappa)$ in $(\bH; \xin, \xout)$ by
$
\mathsf{P}^{(\bH; \xin, \xout)}_{\kappa}.
$
We will in this appendix assume~$\kappa < 8$.

The $\SLE(\kappa)$ in $(\bH; \xin, \xout)$ almost surely avoids any finite collection of points for $\kappa < 8$~\cite{RS-basic_properties_of_SLE}, so the boundary-visiting SLE can only be defined via a suitable approximation procedure. It was motivated (non-rigorously) in~\cite{JJK-SLE_boundary_visits} that the probability to visit the $\varrho$-neighbourhoods of $\hat{x}_1, \ldots, \hat{x}_{N'} $ in the order $\omega$ is of the magnitude $( \varrho^{8/\kappa-1} )^{N'}$ as $\varrho \to 0$, and furthermore
\begin{align*}
\PartF (\xin, \xout) ( \varrho^{8/\kappa-1} )^{-N'} \mathsf{P}^{(\bH; \xin, \xout)}_{\kappa}
[ \gamma \text{ visits } B(\hat{x}_1, \varrho ), \ldots, B(\hat{x}_{N'}, \varrho) \text{ in the order } \omega] 
\end{align*}
should in the limit $\varrho \shrinkto 0$ tend to a positive function of $(\xin, \xout; \hat{x}_1, \ldots, \hat{x}_{N'})$ that satisfies certain M\"{o}bius covariance, PDEs, and asymptotics. Let us call this function the \emph{boundary visit amplitude}.

By~\cite[Theorem~1.1]{KKP}, the functions $\zeta_\omega$ in Theorem~\ref{thm: bdry visit convergence} satisfies the covariance, PDEs, and asymptotics of a boundary visit amplitude, at $\kappa = 2$. Thus, assuming the uniqueness of solutions to such a PDE problem, $\zeta_\omega$ \textit{is} the $\SLE(2)$ boundary visit amplitude. We finish the argument for general $\kappa$, denoting the boundary visit amplitude by $\zeta_\omega$.

Now, modify the limit procedure giving (conjecturally) the boundary visit amplitude \linebreak $\zeta_\omega (\xin, \xout; \hat{x}_1, \ldots, \hat{x}_{N'})$, taking instead $k_i \varrho$-neighbourhoods of each $x_i$, i.e., of different sizes but shrinking at the same rate in the limit $\varrho \to 0$. A guess for such a neighbourhood visit probability would be
\begin{align}
\nonumber
& \mathsf{P}^{(\bH; \xin, \xout)}_{\kappa} 
[ \gamma \text{ visits } B(\hat{x}_1, k_1 \varrho ), \ldots, B(\hat{x}_{N'}, k_{N'} \varrho) \text{ in the order } \omega]
\\
\label{eq: SLE bdry vis proba guess}
& \quad
= ( \varrho^{8/\kappa-1} )^{N'}
\left( \prod_{i=1}^{N'} k_i^{8/\kappa-1} \right) 
\frac{\zeta_\omega (\xin, \xout; \hat{x}_1, \ldots, \hat{x}_{N'})}
{ 
\PartF (\xin, \xout)
}
+ o(( \varrho^{8/\kappa-1} )^{N'}) \qquad \text{as } \varrho \to 0.
\end{align}
Assuming~\eqref{eq: SLE bdry vis proba guess} and using the conformal Markov property of the SLE, one obtains
\begin{align}
\nonumber
& \mathsf{P}^{(\bH; \xin, \xout)}_{\kappa} 
[ \gamma \text{ visits } B(\hat{x}_1, k_1 \varrho ), \ldots, B(\hat{x}_{N'}, k_{N'} \varrho) \text{ in the order } \omega \; \vert \; \mathscr{F}_{t \wedge \tau}]
\\
\nonumber
& \quad
=
\left( \prod_{i=1}^{N'} g'_t(\hat{X}^{(i)}_0)^{8/\kappa-1} \right)
\frac{
\zeta_\omega (W_t, X^{(\mathrm{out})}_t, \hat{X}^{(1)}_t,  \ldots  \hat{X}^{(N')}_t)
}
{ 
\PartF ( W_t, X^{(\mathrm{out})}_t )
}
( \varrho^{8/\kappa-1} )^{N'} + o(( \varrho^{8/\kappa-1} )^{N'})
\\
\label{eq: SLE mgale}
&
\qquad \qquad \qquad
\text{stopped at } \tau.
\end{align}

After all these heuristics, what one can rigorously prove that the leading coefficient of~\eqref{eq: SLE mgale}
\begin{align*}
M_t 
=
\left( \prod_{i=1}^{N'} g'_t(\hat{X}^{(i)}_0)^{8/\kappa-1} \right)
\frac{
\zeta_\omega (W_t, X^{(\mathrm{out})}_t, \hat{X}^{(1)}_t,  \ldots  \hat{X}^{(N')}_t)
}
{ 
\PartF (\xin, \xout)
} ,
\qquad
\text{stopped at } \tau,
\end{align*}
indeed is a positive $\mathscr{F}_t$ local martingale under $\mathsf{P}^{(\bH; \xin, \xout)}_{\kappa}$~\cite{JJK-SLE_boundary_visits, KKP}, i.e., under~\eqref{eq: SLE SDE}. (The drift part in its It\^{o} differential is one of the PDEs defining $\zeta_\omega$, given in~\eqref{eq: bdry vis PDE} for $\kappa =2$.) It is also bounded, by arguments similar to Lemma~\ref{lem: mgale boundedness}, hence a genuine martingale.

Construct now a new probability measure $\mathsf{Q}$ with the Radon--Nikodym derivatives
\begin{align}
\label{eq: RN der defining Q}
\frac{\ud \mathsf{Q}}{\ud \mathsf{P}^{(\bH; \xin, \xout)}_{\kappa}} \big\vert_{\mathscr{F}_t} = M_t/M_0.
\end{align}
By Girsanov's theorem, under the probability measure $\mathsf{Q}$, the process $W$ is up to time $\tau$ governed by
\begin{align}
\label{eq: bdry vis SLE dr fcn}
\ud W_t = \sqrt{\kappa} \ud B_t + \kappa \frac{\partial_{\xin} \zeta_\omega (W_t, X^{(\mathrm{out})}_t, \hat{X}^{(1)}_t,  \ldots  \hat{X}^{(N')}_t)}
{ \zeta_\omega (W_t, X^{(\mathrm{out})}_t, \hat{X}^{(1)}_t,  \ldots  \hat{X}^{(N')}_t) } \ud t ,
\end{align}
i.e., $W$ is the $\SLE(\kappa)$ type process with the partition function $\zeta_\omega$.

Going back to the heuristics, assuming~\eqref{eq: SLE mgale} we interpret~$M_t$ as the conditional probability of $\SLE(\kappa)$ in $(\bH; X^{(\mathrm{in})}_0, X^{(\mathrm{out})}_0)$ to visit $\hat{X}^{(1)}_0, \ldots, \hat{X}^{(N')}_0$. By~\eqref{eq: RN der defining Q}, $\mathsf{Q}$ should thus be interpreted as the measure and~\eqref{eq: bdry vis SLE dr fcn} as the driving function of the $\SLE(\kappa)$ in $(\bH; X^{(\mathrm{in})}_0, X^{(\mathrm{out})}_0)$, conditional on boundary visits at $\hat{X}^{(1)}_0, \ldots, \hat{X}^{(N')}_0$ in the order $\omega$.


%


\bigskip

\bibliographystyle{annotate}

\newcommand{\etalchar}[1]{$^{#1}$}

\end{document}